\documentclass{amsart}
\usepackage{amsmath}
\usepackage{amsfonts}
\usepackage{amssymb}
\usepackage{graphicx}
\usepackage{eso-pic}
\usepackage{color}
\usepackage{wrapfig}
\usepackage{changes}
\usepackage{xspace}
\usepackage{enumerate}

\newtheorem{theorem}{Theorem}[section]
\newtheorem{lemma}[theorem]{Lemma}
\newtheorem{proposition}[theorem]{Proposition}
\newtheorem{corollary}[theorem]{Corollary}

\theoremstyle{definition}
\newtheorem{definition}[theorem]{Definition}

\theoremstyle{remark}
\newtheorem{remark}[theorem]{Remark}

\numberwithin{equation}{section}

\newcommand{\Eclass}{\mathcal{E}}

\newcommand{\Psiint}{\Psi_\mathrm{int}}
\newcommand{\Psiext}{\Psi_\mathrm{ext}}

\begin{document}



\renewcommand{\labelenumi}{($\roman{enumi}$)}

\newcommand{\ind}[1]{1\!\!{\rm I}_{#1}}
\newcommand{\R}{\mathbb{R}}
\newcommand{\N}{\mathbb{N}}
\renewcommand{\L}{\mathrm{L}}
\renewcommand{\d}{\,\mathrm{d}}
\newcommand{\ds}{\displaystyle}
\newcommand{\supp}{\mathop{\rm supp}\,}
\renewcommand{\P}{\mathbb{P}}
\newcommand{\E}{\mathbb{E}}
\newcommand{\e}{\,\mathrm{e}\,}
\newcommand{\dx}{\,\mathrm{d}x}
\newcommand{\dy}{\,\mathrm{d}y}
\newcommand{\dt}{\,\mathrm{d}t}
\newcommand{\dnu}{\,\mathrm{d}\nu}
\newcommand{\dmu}{\,\mathrm{d}\mu}
\renewcommand{\d}{\,\mathrm{d}}
\newcommand{\dist}{\mathrm{dist}}
\newcommand{\C}{\,\mathrm{C}}
\newcommand{\eps}{\varepsilon}
\renewcommand{\leq}{\leqslant}
\renewcommand{\geq}{\geqslant}
\newcommand{\hyp}[1]{{$(\mathbf{#1})$}}
\newcommand{\pe}[2]{(#1\cdot#2)}
\newcommand{\nhyp}[1]{\noindent\hyp{#1}\;}
\newcommand{\ie}{\textit{i.e.}\xspace}
\renewcommand{\H}{\hyp{H}{}\xspace}
\newcommand{\loc}{\mathrm{loc}}
\newcommand{\Fstar}{\ensuremath{\mathcal{F}_\star}\xspace}
\newcommand{\W}{\ensuremath{\mathrm{W}^{1,\infty}}}
\newcommand{\Wloc}{\ensuremath{\mathrm{W}^{1,\infty}_{\rm loc}}}
\newcommand{\Winf}{\ensuremath{\mathrm{W}^{1,\infty}}}
\newcommand{\Bex}{B_{\epsilon^{1/2}}(\bar x)}

\parskip5pt

\setlength{\textwidth}{5.5in} \setlength{\textheight}{8.5in}
 \setlength{\hoffset}{-0.3in}
 \setlength{\footskip}{0.5in}

\newcommand{\tr}{\mathop{\rm Tr}}
\newcommand{\Hess}{H^{\rm ess}}
\newcommand{\Less}{L^{\rm ess}}
\newcommand{\ang}{\mathop{\rm ang}}

\newcommand{\ou}{\overline{u}}
\newcommand{\uu}{\underline{u}}
\newcommand{\ep}[1]{\ensuremath{(\mathrm{EP})_{#1}}\xspace}
\newcommand{\epl}{\ep{\lambda}}
\newcommand{\Slr}{\ensuremath{\mathcal{S}_{\lambda,r}}\xspace}
\newcommand{\csubeta}{\underline{c}_\eta}
\newcommand{\csupeta}{\overline{c}_\eta}

\title[Unbounded solutions of ergodic non-local H-J equations]{On unbounded solutions of ergodic problems for non-local Hamilton-Jacobi
equations}

\author{Cristina Br\"andle}
\address[C. Br\"andle]{Departamento de Matem{\'a}ticas, U.~Carlos III de Madrid,
28911 Legan{\'e}s, Spain}
\email[Corresponding author]{cristina.brandle@uc3m.es}

\author{Emmanuel Chasseigne}
\address[E. Chasseigne]{Institut Denis Poisson UMR 7013, Universit\'e de Tours, Parc de
    Grandmont, 37200 Tours, France }
\email{emmanuel.chasseigne@univ-tours.fr}




\begin{abstract}
    We study an ergodic problem associated to a non-local Hamilton-Jacobi
    equation  defined on the whole space
    $\lambda-\mathcal{L}[u](x)+|Du(x)|^m=f(x)$ and determine whether (unbounded)
    solutions exist or not. We prove that there is a threshold growth of the
    function $f$, that separates existence and non-existence of solutions, a
    {phenomenon} that does not appear in the local version of the problem.
    Moreover, we show that there exists a critical ergodic constant,
    $\lambda_*$, such that the ergodic problem has solutions for $\lambda\leq
    \lambda_*$ and such that the only solution bounded from below, which is
    unique up to an additive constant,  is the one associated to $\lambda_*$.
\end{abstract}

\maketitle

\section{Introduction}
\label{sect:intro}

The starting point of this paper concerns some ergodic problems as they were
studied in \cite{BarlesMeireles2017} and \cite{Ichihara,Ichihara2013}. There,
the authors study the Hamilton-Jacobi equation $$\lambda-\Delta u+H(x,Du)=0,$$
where typically $H(x,Du)=|Du|^m-f(x)$, and $m>1$ (or $m>2$). Our initial
aim was to a consider a simple non-local version of this equation and try
to see how similar results could be obtained: existence of solutions, critical
ergodic constants, and some qualitative behaviour like growth estimates. The
equation {we consider} is the following:
\begin{equation}
    \label{eq:EP}\tag*{$\ep{}$}
    \lambda-\mathcal{L}[u](x)+|Du(x)|^m=f(x)
    \quad\text{in}\quad \R^N,
\end{equation}
where the non-local operator is defined as a convolution with a regular
kernel, $\mathcal{L}[u]:=J\ast u-u$, and $J$ is a continuous, compactly supported
probability density. We explain below why, especially in those ergodic
problems, this equation raises interesting questions and new phenomena, even
compared to the more commonly studied fractional Laplacian, for which the kernel
is singular, $J(x)=1/|x|^{N+\alpha}$, $\alpha\in(0,2)$.

Let us also mention that in the ``standard'' setting, studying ergodic problems
is done either in bounded domains, or in the periodic case, see for instance \cite{BarlesDaLio2005, BarlesSouganidis2001}. Both situations
allow for a better control of the solutions and
{thus, to use compactness arguments}.
The fact that in \cite{BarlesMeireles2017,Ichihara,Ichihara2013}, the authors
consider an unbounded domain (the whole space $\R^N$), with a possibly unbounded
data $f$, leads to various difficulties in the process of constructing
solutions, estimating their growth and getting comparison results.
This  is even more challenging and difficult in our non-local setting.

{We managed to recover most of the
results one can expect for $\ep{}$, but this work turned out to be much more
interesting and demanding than a simple
adaptation from the local case.} We had to develop new methods and techniques to
deal with the non-local term, and we found out that there are natural
limitations in the growth of solutions and the right-hand side $f$,
a feature which is not present in the local setting. We are
also convinced that the ideas that we use here can be helpful in the
local case as well, and improve some of the results found
in~\cite{BarlesMeireles2017,Ichihara,Ichihara2013}.

By a  solution of~\ep{} we understand a pair $(\lambda,u)$ where $\lambda\in\R$
and $u$ is a continuous viscosity solution of the equation.
We also refer to $\ep{\lambda}$ when $\lambda$ is given and the unknown is only $u$.
Observe that~\ep{} is invariant by addition of constants to the solution, as is
usually the case in ergodic problems. As we shall see, the solutions will be
actually locally Lipschitz continuous so that the equation will hold almost
everywhere and in the weak sense.

This kind of  ergodic problems are known,~\cite{Ichihara2013}, to be closely
related to the asymptotic behaviour of solutions of the associated evolution
equation, which would read in our case
\begin{equation}\label{ergo.time}
  u_t-\mathcal{L}[u]+|Du|^m=f\quad\text{in}\quad\R^N.
\end{equation}
It is not the purpose of this paper to investigate this question, but let us just
mention that in general, solutions of \eqref{ergo.time} are expected to
behave like $u(x,t)=\bar\lambda t+v(x)+o(1)$ as $t\to\infty$ where $(\bar\lambda,v)$
is a solution of \ep{}.
The specific value $\bar\lambda$ is usually obtained by
taking the supremum of all the $\lambda$'s such that there is a solution $v$ of
\ep{\lambda}. And as for $v$, it is in general the unique (up to an additive
constant) solution of $\ep{}$ which is bounded from below. In this paper we
focus on the existence and properties of this pair $(\bar\lambda,v)$, {not on the asymptotic behaviour for \eqref{ergo.time}.}

\subsection{The framework}

Throughout the paper, we make the following fundamental assumptions:\\[2mm]
  \noindent $(i)$ the kernel $J:\R^N\to\R$ is $\C^1$, symmetric, radially
decreasing, compactly supported in $B_1(0)$, with  $\int J(y)\dy=1$ and
strictly positive in all $B_1(0)$;\\[2mm]
$(ii)$ we restrict ourselves to the superquadratic case $m>2$ (more on this
below);\\[2mm]
$(iii)$ the function $f$ is assumed to be at least continuous
    and bounded from below.

\noindent We will give more precise assumptions on $f$ later, but these basic three
assumptions $(i)$--$(iii)$ will always hold and we shall not recall them anymore.

\noindent{\sc On the non-local term ---} We use here a convolution with a
regular, compactly supported kernel. This choice has several consequences that
need to be dealt with. First, no regularizing effect can be obtained from
operator $\mathcal{L}$. Indeed,  contrary to
the Laplacian which is a second-order one, or a fractional Laplacian which would
be of order $\alpha\in(0,2)$, {operator $\mathcal{L}$ can be seen} as a
zero-order term. However, this non-local operator still enjoys some strong
maximum principle. We refer to \cite{AndreuMazonRossiToledo,
    ChasseigneChavesRossi2006} and the references therein for general properties
of this type of non-local operator. In particular, $\mathcal{L}$ is known to be
an approximation of the Laplacian as the support of $J$ shrinks to the origin
(when $J$ is symmetric).

We also choose  to consider here a compactly supported kernel.

When the non-local operator is defined through a fractional Laplacian, for instance, the tail of
order $1/|x|^{N+\alpha}$ implies a power-type growth restriction for all
possible solutions, since $\mathcal{L}[u]$ has to be defined, at least.
On the
contrary, if $J$ is compactly supported, $\mathcal{L}[u]$ is always defined, as
long as $u$ is locally bounded. But (see below), contrary to the local case, the
presence of a non-local term in the equation implies some growth limitation,
even in the case of a compactly supported kernel. A similar behaviour was also
found in~\cite{BrandleChasseigneFerreira}, where the growth of the initial data is
limited, and differs from the local heat equation.

\noindent{\sc On the Hamiltonian ---} In this paper we consider
the case $H(x,p)=|p|^m-f(x)$, but most of the results are adaptable
to more general cases, for instance $H(x,p)=a(x)|p|^m-f(x)$ where $a(x)$ is
regular and does not degenerate, as is done in~\cite{Ichihara}.
Notice however that since our
solutions are not necessarily bounded, the gradient is only locally bounded in
principle, and as is well-known in Hamilton-Jacobi equations, mixing the
$x$-dependence with the $p$-dependence leads to several difficult issues.

As we mentioned, we will restrict ourselves to the {superquadratic} case, $m>2$. Actually, there is only one place
where this specific condition seems to play a role, namely in the existence
construction (proof of Proposition~\ref{proposition.existence.vR.R.varepsilon}),
when using results of~\cite{CapuzzoLeoniPorretta} to deal with a
\textit{vanishing viscosity
    approximation}. It is not clear to us whether this restriction
is purely technical and could be relaxed to the more general assumption $m>1$. This is a clearly difference from the local equation, since in this case, the viscous term does not vanish.

\subsection{Main results} We present now the main results of this paper,
which can be summarized in three items.

\noindent\textsc{A. Growth limitations ---} As we said, since $J$ is
compactly supported, the non-local term $\mathcal{L}[u]$ is well-defined for any
locally bounded function $u$. However, contrary to the local case, it turns out
that problem \ep{} is not solvable for arbitrary growths of function $f$. More
precisely, if $f(x)=C\exp(a^{|x|})$, then \ep{} is solvable only when $a\leq m$
(at least in the class of radial, radially increasing at infinity solutions).

The formal explanation is that in order to solve the equation, the $|Du|^m$-term
has to be the leading term. But, at least for fast growing radial supersolutions
$\psi$, this is not the case: the convolution looks like
$-\mathcal{L}[\psi](r)\simeq -\psi(r+1)$, which grows faster than $|\psi'|^m(r)$
and this implies that the supersolution inequality cannot be satisfied.

Similarly, when solutions exist they cannot have arbitrary growth for the
    same reason (and again we only have non-existence results in the class of radial
    and radially increasing at infinity solutions).
    We refer to Lemma~\ref{lem:non.existence} and
    Corollary~\ref{cor:non.existence} for precise statements.

\noindent\textsc{B. Existence of solutions and of a critical ergodic
    constant ---} We prove that for functions $f$ in a suitable growth class,
typically  $f(x)\leq C\exp\big(m^{|x|}\big)$ for some $C>0$, problem \ep{} is
solvable. Moreover,  there is a bound for the constructed solution,
$u(x)\leq\Psi(x):=|x|f^{1/m}(x)$ for $|x|$ large, so that $u(x)\leq
C|x|\exp\big(m^{|x| -1}\big)$ for large~$|x|$.

Getting this existence result requires to deconstruct all the methods that are
used in \cite{BarlesMeireles2017,Ichihara,Ichihara2013} (and even
\cite{GT}, see Appendix). A big issue that we face is   that we do not have an universal local gradient estimate, as it is the case in~\cite{BarlesMeireles2017,Ichihara,Ichihara2013}.
This is due to the fact that $\mathcal{L}$ is just a
zero-order operator. We manage to bypass this difficulty by using a supersolution
(obtained by a modification of function $\Psi$ above)
in order to control the non-local term.
But this implies several technical problems, since $\Psi$ is only
a supersolution of~\ep{} for $|x|\neq0$, see the whole construction in Section~3.

Notice that there are bigger supersolutions, but this specific $\Psi$ yields a
kind of \textit{minimal supersolution} in the sense that bounded from below
solutions behave like $\Psi$ (see below).

Once the existence result is proved, it is usual to consider the critical
ergodic constant as the supremum of all $\lambda$'s such that $\ep{\lambda}$ is
solvable. However, we still face the difficulty of estimates here and we need to
include {a limiting upper behaviour} of solutions in the definition:
$$\bar\lambda:=\sup\Big\{\lambda\in\R:\text{ $\exists u$, solution of
        $\ep{\lambda}$, such that }
    \limsup_{|x|\to\infty}\frac{u(x)}{\Psi(x)}<\infty\Big\}.$$
We prove that $\bar\lambda$ is finite and that there exists a solution $u$ associated
to $\bar\lambda$.
Again, it is natural to have a limitation for the growth of solutions. Indeed, if
$u$ grows too fast, then $\mathcal{L}[u]$ becomes the leading term of the
equation, and we have seen already that this is not possible if we want to have a solution.

\noindent\textsc{C. Characterization of the critical ergodic constant ---}
As in the local case, we prove that the critical ergodic constant $\bar\lambda$
can be characterized by the fact that $(\lambda,u)$ is a solution of \ep{}
such that $u$ is bounded from below and $\limsup u(x)/\Psi(x)<\infty$ if and
only if $\lambda=\bar\lambda$. And in this case,
$u$ is uniquely determined (up to an additive constant).

Notice that in \cite{BarlesMeireles2017}, such results are obtained for
solutions  and functions $f$ which grow like powers. In contrast, we are able to consider much faster
growths, like $\exp(a^{|x|})$.
A key step in this improvement is to prove a bound from below for solutions such
that $\inf u>-\infty$. This is done in Lemma~\ref{lemma:u.bounded.below.Psi}.
This Lemma is a refinement of \cite[Proposition 3.4]{BarlesMeireles2017}, and allows
to treat faster growths. Actually, this approach could also be applied to the
local case in order to generalize various results in \cite{BarlesMeireles2017}.

\subsection{Organization of the paper} In Section~\ref{sect:ass.pre} we
state the main hypotheses on the function $f$, construct sub and supersolutions
to the problem and prove {non-existence results which illustrate the fact that
    the growths of $u$ and $f$ are limited}. Section~\ref{sect:approximate}
deals with auxiliary problems defined on a bounded domain.  Then, in
Section~\ref{sect:existence} we prove the existence of solutions of~\ep{}.  The
last four sections are devoted to the critical ergodic constant and bounded from
below solutions. In particular, we establish the existence of a critical ergodic
constant in Section~\ref{sect:critical} under some growth restriction. In
Section~\ref{sect:bounded} and~\ref{sect:uniqueness} we prove that there are
solutions that are bounded form below, that these are unique (up to an additive
constant) and that they are associated to the critical ergodic constant.
Finally, in Section~\ref{sect:revisited} we extend the class of
solutions associated with the critical ergodic constant and prove some
continuous dependence of the critical ergodic constant with respect to~$f$.

\


\section{Preliminaries}
\label{sect:ass.pre}

\noindent\textsc{Basic Notations ---} In the following, $B_R$ will stand for $B_R(0)=\{x\in\mathbb{R}^N : |x|<R\}$ and we use the notation $|x|\gg1$
to say that a property is valid for $|x|$ sufficiently large.

We will denote $u(x)=o(v(x))$ to say that $u(x)/v(x)\to0$ as $|x|\to\infty$. In
particular, $o_\alpha(1)$ represents a quantity which goes to zero as the
parameter $\alpha$ goes to zero (or $+\infty$, depending on the situation). If
some uniformity with respect to some other parameter is required, this will be
mentioned explicitly.

\subsection{Definitions and hypotheses}

\begin{definition}
\label{def:viscosity}
	A locally bounded u.s.c. function $u:\mathbb{R}^N\to\mathbb{R}$
 is a viscosity subsolution of
	\ep{} if for any $\C^1$-smooth function $\varphi$, and any point
	$x_0\in\mathbb{R}^N$ where $u-\varphi$ reaches a maximum, there holds,
	$$
\lambda-\mathcal{L}[u](x_0)+|D\varphi(x_0)|^m-f(x_0)\leq 0.
    $$
\end{definition}

A locally bounded l.s.c. function is a viscosity supersolution if the same holds
with reversed inequalities and the maximum point replaced by a minimum. Finally
a viscosity solution is {a continuous function $u$ which is at the same time
    a sub- and a super-solution of \ep{}.}

Notice that, in the above definitions we only need the test function
$\varphi$ to be $\C^1$ in a neighborhood of $x_0$ (or even only at $x_0$), and we
shall use this remark when we use test functions which are not $\C^1$ in all
$\R^N$.

\begin{remark}
\label{remark:viscosity.bounded}
  If we consider $u:\Omega\to\mathbb{R}$ and the equation defined on a bounded
  domain $\Omega$ together with a boundary condition, say $u=\psi$ on
  $\partial\Omega$, then the definition of viscosity subsolution (respectively
  supersolution) has two parts, depending whether the maximum point $x_0$ is
  achieved inside $\Omega$ or on the boundary, $\partial\Omega$. In this latter
  case, the condition for $u$ to be a  subsolution reads
  $$
   \max(\lambda-\mathcal{L}[u](x_0)+|D\varphi(x_0)|^m-f(x_0), u(x_0)-\psi(x_0))\leq 0
  $$
  (respectively $\min$ and $\geq$ for supersolutions). However, we shall not
  use boundary conditions here: we have only a viscosity solution in a ball
  $B_R$ and send $R\to\infty$, see Lemma~\ref{lemma:epsilon.to.0}.
\end{remark}

We list now the complete set of assumptions we use throughout this paper.  We
want to stress at this point, that we could have simplified this list by making
strong assumptions (for instance, assuming that $f$ is radial and radially
increasing). However, we opted to keep track of what was really necessary to
assume to produce each result. We think that the methods we design, with these
weak assumptions, can be helpful in other situations. We comment on each
hypothesis and give typical examples.

The main hypothesis on the right hand-side $f$, that we use throughout the paper is the following:
\begin{enumerate}[\noindent\bf(H1)]
    \setcounter{enumi}{-1}
    \item
      $f:\R^N\to\R$ is $\C^1$ and $\inf\{f(x):|x|=r\}\to\infty$ as
      $r\to\infty$.
\end{enumerate}
In particular, we are assuming that $f$ is \textit{uniformly coercive} so that
for $x$ large enough we have  $f>0$ and we can set
\begin{equation*}\label{def:Phi}
    \Phi(x):=|x|f(x)^{1/m}.
\end{equation*}

In addition we have to impose some extra hypothesis on $f$. The next set of
assumptions are related to its growth. The first two hypotheses,~\hyp{H1}
and~\hyp{H2}, are required to construct a supersolution in
$\R^N\setminus\{0\}$. Hypothesis \hyp{H1} is fundamental here, it is were we see
the limitation in the growth of $f$, see more below.

\begin{enumerate}[\noindent\bf(H1)]
    \item For $|x|\gg1$, $\sup\limits_{y\in B_1(x)} |D\Phi(y)|\leq
        |D\Phi(x)|^m$.\\
    \item For $|x|\gg1$, $x\cdot Df(x)\geq -f(x)$.
\end{enumerate}

The following two hypotheses,~\hyp{H3} and~\hyp{H4}, control the behaviour of
$f$ from below. This is crucial in order  to prove that solutions which are
bounded from below, actually have a minimal behaviour at infinity, which is
given by $\Phi$.
\begin{enumerate}
  [\noindent\bf(H1)]
\setcounter{enumi}{2}
	\item As ${|x|\to\infty}$, $\Phi(x)=o(f(x)).$\\
        \item There exists $\eta_0\in(0,1)$ such that for all
            $\eta\in(0,\eta_0)$, there exists $\csubeta,\csupeta>0$ and
            $R_\eta>0$ such that\\[6pt]
            $\text{for $|x|\geq R_\eta$, and $s\in
            B_\eta(0)$,\quad} \csubeta f\big((1-\eta)x\big)\leq
            f\big(x+s|x|\big)\leq \csupeta f\big((1+\eta)x\big).$
 \end{enumerate}
The last set of hypotheses is needed in order to get a comparison result
in the class of bounded from below solutions. Depending on whether we are in the
``slow'' (power-type) case or ``fast'' (exponential or more) case, the
approach differs a little bit,  but we cover both cases.

\begin{enumerate}
  [\noindent\bf(H1)]
\setcounter{enumi}{4}
    \item There exists $a_0>1$: $\forall a\in(1,a_0)$, as $|x|\to\infty$,
	$\sup\limits_{|z|\leq 1}\Phi(a(x+z))\ll f(x)$.\\

    \item There exists $R_0>0$, such that $\forall a\in(1,a_0)$ and $|x|>R_0$,
	$f(ax)\geq af(x)$.\\

    \item One of the following holds:\\
        \textit{Slow case --} for all $a\in(1,a_0)$, $\limsup
        (f(ax)/f(x))<\infty$ as $|x|\to\infty$;\\
        \textit{Fast case --} for all $a\in(1,a_0)$, $\liminf
        (f(ax)/f(x))=+\infty$ as $|x|\to\infty$.
\end{enumerate}

\noindent\textsc{Examples and discussion on the hypotheses ---} As we said,
\hyp{H1} highlights the limiting behaviour
of $f$ in order to get a solution. It can be checked that
if $f(x)=\exp\big(p^{|x|}\big)$, then \hyp{H1} is satisfied if and only if
$p\leq m$. This hypothesis is essential in order to control the non-local term
by the gradient term.

On the contrary,~\hyp{H3} and~\hyp{H4} both imply that $f$ has a minimal growth.
By the specific form of $\Phi$, \hyp{H3} is equivalent to $f(x)\gg |x|^{m _*}$
where $m_*:=m/(m-1)$. Actually, this is not a real limitation since if $f$
does not grow so fast, the methods in \cite{BarlesMeireles2017} readily adapt.

Hypothesis~\hyp{H5}, though it is similar to~\hyp{H3}, {is a bit more
    restrictive than~\hyp{H3}.}
For power-type functions $f$, this condition also reduces to $f(x)\gg|x|^{m_*}$
and  \hyp{H5} implies \hyp{H3}.

Finally,~\hyp{H2},~\hyp{H4} and~\hyp{H6} are automatically satisfied if $f$ is
radial and increasing. Hence, these hypotheses are needed to control how much
the function $f$ is allowed to deviate from this behaviour. Notice though, that
they allow $f$ to be quite \lq\lq far\rq\rq\  from radial and increasing.

In conclusion, the typical functions that satisfy all these hypotheses are the
following:
$$\begin{aligned}
f_1(x)&=c|x|^\alpha\text{ with $\alpha>m_*$,}\\
f_2(x)&=c\exp(\alpha |x|)\text{ with $\alpha>0$,}\\
f_3(x)&=c\exp\big(p^{|x|}\big)\text{ with $p\leq m$.}
\end{aligned}
$$
Some non-radial as well as some non-monotone versions are allowed within
the range of~\hyp{H2},~\hyp{H4} and~\hyp{H6}.

Hypothesis \hyp{H7} covers all cases $f_1,f_2,f_3$ above, but makes the
distinction between \textit{power-type} growths which satisfy \hyp{H7}-slow and
exponentials for which \hyp{H7}-fast is fulfilled.

\begin{remark}
\label{rem:regularity.f}
    We assume that $f$ is $\C^1$ throughout the paper for simplicity: with this
    assumption we can compute and use the gradient of
    $\Phi(x)=|x|f^{1/m}(x)$, for $|x|$ large.
       In fact, regularity of
    $f$ is not an issue here and we could consider only continuous functions to
    get exactly the same results by using smooth approximations of $f$.
\end{remark}

Across Sections 4--6 we will assume that $f$
verifies the three assumptions \hyp{H0}--\hyp{H2} without
mentioning it anymore. In Section 7,
where we prove uniqueness, we will use more assumptions on $f$ and hence we will write  only what is really necessary in order to prove each result.

\subsection{Subsolutions and Supersolutions}
It is straightforward to see that, for
$\lambda\leq\min(f)$, typical subsolutions of \epl are the constants.
In this range of $\lambda$-values, {there are also coercive subsolutions, as the
following lemma shows. Those subsolutions will help us build solutions which
tend to infinity at infinity (see Section~\ref{sect:bounded}).}

\begin{lemma}\label{lem:subsol}
    Let $f$ verify \hyp{H0}. Then for any $\lambda\leq\min(f)$ there exists a
    Lipschitz subsolution $\Theta_\lambda$ of \epl, such that
    $\Theta_\lambda(x)\to\infty$ as $|x|\to\infty$.
\end{lemma}
\begin{proof}
    Since $\lambda\leq\min(f)$ and $f$ is uniformly coercive,
    there exists $R_*>0$ which depends on $\lambda$ such that $f(x)\geq\lambda+\kappa^m$, if $|x|\geq
    R_*$, for some $\kappa>0$.
        We define
    \begin{equation*}
      \label{eq:def.subsolution.Theta}
      \Theta_\lambda(x):=\kappa(|x| - R_*)_+
    \end{equation*} which is (globally)
    Lipschitz. Using Lemma~\ref{lem:convex} we see that for any $x\in\R^N$,
    $-\mathcal{L}[\Theta_\lambda](x)\leq 0$. Moreover, since
    $|D\Theta_\lambda|=\kappa$ or $|D\Theta_\lambda|=0$ almost everywhere, we get in any case
    \begin{equation}
      \label{eq:theta.condition}
      \lambda-\mathcal{L}[\Theta_\lambda]+|D\Theta_\lambda|^m\leq
    \lambda+\kappa^m\leq f.
    \end{equation}

    Notice that the exact proof has to be done in the sense of viscosity, but at the
    points where $|x|=R_*$, no testing is done for the subsolution condition, while
    at other points the function is smooth. So, $\Theta_\lambda$ is indeed a
    coercive subsolution in $\R^N$, {in the sense of viscosity}.
\end{proof}

\begin{remark}
    The parameter $\kappa$ is somewhat free if we allow $R$ to be big. {More
    precisely, for $|x|$
    big enough $f(x)$ is big and we can choose $\kappa$ big, thus we
    can build subsolutions with arbitrary big linear growth.}
\end{remark}

If we try now to construct a supersolution, the first thing that we face is that
it is not possible to do it in all $\R^N$, if $\lambda\leq\min(f)$.

\begin{lemma}
  \label{lemma:non.existence.super}
 {
  Let $f$ verify \hyp{H0} and  $\lambda\leq\min(f)$. Then there is no
  coercive, l.s.c. viscosity supersolution of~\ep{} in all $\R^N$.
  }
\end{lemma}

\begin{proof}
 Assume by contradiction that there is such a supersolution, $\psi$. Since it is
 lower semi-continuous and coercive, $\psi$ reaches its global minimum at some
 point $x_0$. Then, since $\psi(y)\geq \psi(x_0)$ for all
 $y\in\R^N$, we can use the constant $\psi(x_0)$ as a test-function for the
 viscosity inequality at $x_0$, which yields
 $$
 0\leq f(x_0)-\lambda\leq -\mathcal{L}[\psi](x_0)-|D(\psi(x_0))|^m=
 -\int_{B_1(x_0)}J(x_0-y)(\psi(y)-\psi(x_0))\d y\leq 0.
 $$
 This implies that necessarily, $\psi(y)=\psi(x_0)$ for all $y\in B_1(x_0)$.

  Then we can repeat the argument using as center any $y\in B_1(x_0)$ and we
  finally get that $\psi(y)=\psi(x_0)$ for $y\in\mathbb{R}^N$.  But then
  we get a contradiction since $\psi$ cannot be constant.
 
 \end{proof}

However we are able to build supersolutions in
$\R^N\setminus\{0\}$ (without any restriction on $\lambda$). In order to do it,
we look first at~\ep{0}, \begin{equation}
    \label{eq:EP0}
    \tag*{$\ep{0}$}
    -\mathcal{L}[u](x)+|Du(x)|^m=f(x)
    \quad\text{in}\quad \R^N.
\end{equation}

In this construction we assume that \hyp{H0}--\hyp{H2} hold, so there exists
$R^*>0$ such that for any $|x|\geq R^*$, $f(x)\geq1$ and \hyp{H1},\hyp{H2} hold
for such $x$.

\begin{remark}
  \label{rem:R_*=R^*} If we take $\kappa=1$ and $\lambda=0$ in the construction of the subsolution $\Theta_\lambda$, then $R_*=R^*$. We will use this fact in Section~\ref{sect:bounded} in order to construct bounded from below solutions.
\end{remark}

Up to fixing the constants, we use the
following construction: we set $\Psiint:=b|x|$ for $|x|<R^*+1$;
$\Psiext=c\Phi$ for $|x|>R^*$ and then combine $\Psiint$ and $\Psiext$ in the
intermediate region $R^*\leq |x|\leq R^*+1$, in order to get a supersolution
of~\ep{0} for all $|x|\neq 0$. We finally define $\Psi_\lambda=c_\lambda\Psi$ which
yields a supersolution to~\ep{\lambda} for $|x|\neq0$, provided $c_\lambda$ is
well-chosen.

\begin{lemma}
  \label{lemma:psi1.super}
  There exists $c_0>0$ such that for any $c\geq c_0$,
  $\Psiint(x):=c|x|$ is a supersolution of~\ep{0} for
  $0<|x|\leq R^*+1$.
\end{lemma}

\begin{proof}
  The proof is straightforward: we first have
  $$
  D\Psiint(x)=c\frac{x}{|x|}\mbox{\quad and \quad} |\mathcal{L}[\Psiint](x)|\leq
  \sup_{y\in B_1(x)}|D\Psiint(x+y)|=c.  $$
  Then, since $m>1$, for any $c$ big enough we have
  \begin{equation*}
    \label{eq:condition.c0}
    c^m-c\geq\max\limits_{B_{R^*+1}}f,
  \end{equation*}
  and we get
  $
  -\mathcal{L}[\Psiint](x)+|D\Psiint(x)|^m \geq -c+c^m\geq f.
  $
\end{proof}

\begin{lemma}
\label{lemma:psi2.super}
  Let $f$ verify {\rm\bf(H0)--(H2)}. There exits
  $c_1>0$ such that for any $c\geq c_1$,
  $\Psiext=c\Phi$ is a supersolution of~\ep{0} for $|x|\geq R^*$.
  \end{lemma}

\begin{proof}
  We estimate each term in~\ep{0} separately. On one hand we have
  $$
  D\Psiext=cD\Phi=c\Big(\frac{x}{|x|}f^{1/m}+\frac{|x|}{m}f^{1/m-1}\cdot
  Df\Big).$$ Using \hyp{H2},
  \begin{equation}
  \label{eq:estimate.below.psi1}
  \frac{x}{|x|}\cdot D\Phi= \frac{x}{|x|}\cdot \frac{x}{|x|}\Big(f^{1/m}+\frac{1}{m}f^{1/m-1}x\cdot Df\Big)
  \geq f^{1/m}\Big(1-\frac1{m}\Big),
   \end{equation}
    from where we get that $|D\Phi|\geq |\frac{x}{|x|}\cdot D\Phi|\geq f^{1/m}(1-\frac1{m}).$

  On the other hand, in order to estimate the non-local term we use \hyp{H1} and
  get
  $$
  |\mathcal{L}[\Phi](x)|\leq \sup_{y\in B_1(x)}|D\Phi(x+y)|\leq |D\Phi(x)|^m.
  $$
  Therefore, if $c$ is such that
  \begin{equation*}
    \label{eq:condition.c1}
    c^m-c\geq \frac{m}{m-1},
  \end{equation*}
  we get
  $$
  -\mathcal{L}[\Psiext](x)+|D\Psiext(x)|^m \geq -c|D\Phi(x)|^m+c^m|D\Phi(x)|^m\geq \frac{m-1}{m}(c^m-c)f\geq f,
  $$
  and conclude that $\Psiext$ is a supersolution of~\ep{0} for $|x|\geq R^*$.
\end{proof}

Finally, the construction ends by interpolating between $\Psiint$ and $\Psiext$ in
the region $R^*\leq|x|\leq R^*+1$. To this aim, let
$\chi:[0,\infty)\mapsto[0,\infty)$ be a regular,
radial and non-decreasing function that verifies  $\chi(r)=0$ for $r\leq R^*$
and $\chi(r)=1$ for $r\geq R^*+1$. We set
\begin{equation}
  \label{eq:def.psi.0}
\Psi:=(1-\chi)\Psiint+\chi\Psiext\quad \text{in}\quad \R^N.
\end{equation}

\begin{lemma}
\label{lemma:psi.chi.super}
  Let $f$ verify \hyp{H0}--\hyp{H2}. There exists $c_2>0$ such that
  for any $c\geq c_2$, $\Psi$ is a supersolution of~\ep{0} for
  $R^*\leq|x|\leq R^*+1$.
\end{lemma}

\begin{proof}
    We first give a rough estimate of the non-local term. Notice that since
    $|x|\geq R^*$, $f(x)\geq1$ so that $\Psiext(x)\geq c|x|\geq0$. Since also
    $\Psiint(x)=c|x|\geq0$,
    it follows that for any $R^*\leq|x|\leq R^*+1$, $\Psi(x)\geq0$. Hence, for
    such $x$,
    $$-\mathcal{L}[\Psi](x)\geq -(J\ast\Psi)(x)\geq
    -c(R^*+1)-c\,\sup_{B_{R^*+2}\setminus B_{R^*}}\Phi=-Kc
    $$
    for some constant $K$ depending only on $f$ (through $R^*$ and $\Phi$).

    We now turn to the gradient term.
    As we noticed, $\Psiext(x)\geq c|x|=\Psiint(x)$ for $R^*\leq|x|\leq R^*+1$.
    And since $\chi$ is radially nondecreasing, we get, using in the last line \eqref{eq:estimate.below.psi1} and the fact that $f(x)\geq1$ for $|x|\geq R^*$,
    $$\begin{aligned}
    \frac{x}{|x|}D\Psi(x) & =(1-\chi)\frac{x}{|x|}D\Psiint+
    \chi\frac{x}{|x|}D\Psiext+\chi'(\Psiext-\Psiint)\\
     & \geq (1-\chi)\frac{x}{|x|}D\Psiint+
     \chi\frac{x}{|x|}D\Psiext\\
     &\geq  (1-\chi)c+\chi c (1-1/m)= c(1-\chi/m)\geq c(1-1/m).
    \end{aligned}
    $$
    This gives a lower estimate for the gradient of $\Psi$: for any
    $R^*\leq|x|\leq R^*+1$,
    $$\begin{aligned}
    |D\Psi(x)|^m &\geq \big|\frac{x}{|x|}D\Psi(x)\big|^m\geq
    {c^m(1-1/m)^m}.
        \end{aligned}
    $$

    To conclude, we get that for any $R^*\leq|x|\leq R^*+1$,
    $$-\mathcal{L}[\Psi]+|D\Psi|^m-f\geq -Kc+{c^m(1-1/m)^m}-\sup_{B_{R^*+1}\setminus
        B_{R^*}} f.$$
    Hence, if $c$ is big enough, since $m>1$,
    we obtain that the right-hand side is non-negative which yields the result.
\end{proof}

We are now ready to construct a supersolution for \ep{\lambda} for $|x|\neq 0$.
We first fix $c_*=\max(c_0,c_1,c_2)$ where $c_0,c_1$ and $c_2$ are defined in
the lemmas above. Then the corresponding function $\Psi$ is $\C^1$-smooth and it
is a supersolution of \ep{0} in $\R^N\setminus\{0\}$. In order to deal with a
non-zero ergodic constant $\lambda$, it is enough to multiply $\Psi$ by some
constant (depending on $\lambda$). We set
\begin{equation}
  \label{eq:super.lambda}
  \Psi_\lambda:= c_\lambda\Psi,
  \quad \text{ where }c_\lambda=(2+\lambda^-)\text{ and }\lambda^-=\max(0,-\lambda)\geq0.
\end{equation}

\begin{proposition}
\label{prop:psilamba.strict}
  Let $f$ verify \hyp{H0}--\hyp{H2}, then $\Psi_\lambda$ is a strict
  supersolution for \ep{\lambda} for all $|x|\neq 0$.
\end{proposition}

\begin{proof}
    Recall that for $|x|> R^*$, $f(x)\geq 1$ and since $\Psi$ is a
    supersolution of \ep{0}, for such $x$ we have
  $$
  \begin{aligned}
   \lambda -\mathcal{L}[\Psi_\lambda]+|D\Psi_\lambda|^m&=\lambda
   -(2+\lambda^-)\mathcal{L}[\Psi]+(2+\lambda^-)^m|D\Psi|^m\\
    &\geq \lambda+(2+\lambda^-)\big(-\mathcal{L}[\Psi]+|D\Psi|^m\big)\\
   &\geq\lambda+(2+\lambda^-)f\\
   &\geq \lambda+\lambda^-+1+f\geq f+1.
  \end{aligned}
    $$
  On the other hand, if $|x|\leq R^*$, $\Psi_\lambda=(2+\lambda_-)\Psiint$.
  Then, as in Lemma~\ref{lemma:psi1.super} we obtain
  $$\begin{aligned}
      -\mathcal{L}[\Psi_\lambda]+|D\Psi_\lambda|^m & \geq
  \lambda+(2+\lambda^-)^m c^m-(2+\lambda_-)c\geq
  \lambda+(2+\lambda^-)(c^m-c)\\
  & \geq\lambda+ (2+\lambda^-)\max_{B_{R^*+1}}f\geq f+1.
  \end{aligned}
  $$
  Notice that in the last inequality, we use that the maximum of $f$
  on $B_{R^*+1}$ is greater than or equal to one since at least
  $f(x)\geq1$ for $|x|\geq R^*$.
\end{proof}

\begin{remark}
\label{rem:propoperties.Psi}
  For $|x|\gg 1$, $\Psi_\lambda=(2+\lambda_-) \Phi$. Hence,
  the supersolution $\Psi_\lambda$ {satisfies the same hypotheses that we assume
      on $\Phi$.}
\end{remark}

\begin{remark}
\label{rem:supersolution.bigger.constant}
  Notice that for any $c\geq c_\lambda$, $c\Psi$ is also a strict supersolution of~\ep{\lambda}.
\end{remark}
%

\newcommand{\Rinf}{\mathcal{R}_\infty}

\subsection{Non-Existence results}
\label{sect:non.existence} We end this section by showing, at least in a radial
case, that (super)solutions to~\ep{} only exist if $f$ does not grow too fast.
And in that case, (super)solutions cannot grow too fast either. To formulate the
result, let us introduce the class
$$\Rinf:=
\big\{\phi:\R^N\to\R\text{ such that for $|x|\gg1$, $\phi$ is radial
        and radially increasing}\big\}.$$

\begin{lemma}\label{lem:non.existence}
    Let $\alpha,\beta,r_0,m>1$, $\eps\in(0,1)$ and consider the following problem
    \begin{equation}\label{ineq:lem.nonex}
        \begin{cases}\text{find  $\psi\in \C^1\cap\Rinf$ such that } \\
        \alpha\psi(r)-\beta\psi(r+1-\eps)+(\psi'(r))^m\geq
        f(r)\text{ for }r>r_0,
    \end{cases}
    \end{equation}
    where $f\in\Rinf$ satisfies \hyp{H0}. Given $a>m$, there exists
    $\eps_0(a,m)\in(0,1)$ such that:

    \noindent $(i)$ there is no solution $\psi$ of \eqref{ineq:lem.nonex} if
    $f(r)\geq c\exp(a^{r})$ for $r\gg1$.

    \noindent $(ii)$ there is no solution $\psi$ of \eqref{ineq:lem.nonex}
    satisfying $\psi(r)\geq c\exp(a^{r-1})$ for $r\gg 1$.
\end{lemma}
\begin{proof} The proof is essentially the same for both results, but we start with
    $(ii)$, since $(i)$ will need only a small preliminary step before applying
    the same method. We proceed by contradiction, assuming that a function
    $\psi$ satisfies \eqref{ineq:lem.nonex} and $\psi(r)\geq c\exp(a^{r-1})$ for
    some $C>0$ and $a>m$. The following computations will be done for $r$ big
    enough so that $f$ and $\psi$ are radial and radially increasing.

    We first
    claim that  $\psi(r)\to +\infty$ as $r\to\infty$. Indeed, since
    $\psi$ is nondecreasing, it is bounded from below on $[r_0,+\infty)$ and if,
    in addition, we assume that it is bounded from above, we get that
    for some constant $C>0$, $(\psi'(r))^m\geq f(r)-C$. Hence,
    $\psi'(r)\geq1$ for $r$ big enough (recall that $f$ satisfies \hyp{H0}),
    which is a contradiction with the boundedness of $\psi$. Hence $\psi$ is not
    bounded and since it is nondecreasing, the claim holds.

    Then, we can assume that for some $r_1>r_0$, $\psi(r)\geq 1/\alpha$ on
    $[r_1,\infty)$, which implies that for $r\geq r_1$,
     \begin{equation}\label{ineq:nonex.2}
    (\alpha\psi(r)+\psi'(r))^m\geq f(r)+\beta\psi(r+1-\eps).
    \end{equation}
    From this inequality, we prove by an iteration process
    that $\psi$ has to blow-up for any $r$ big enough, a contradiction. For
    $r$ big enough we can assume that $f(r)\geq0$, hence we only use the
    $\beta\psi$-term in the right and side:
    \begin{equation}\label{eq:ode}
        \forall r>r_1\,,\quad
        (\alpha\psi(r)+\psi'(r))^m\geq \beta c\exp(a^{r-\eps})\,.
    \end{equation}

    We claim that there exists $\eps_0(a,m)>0$ such that if
    $0<\eps<\eps_0$, there exists $\delta\in(\eps,1-\eps)$ such that
    $\gamma:=a^{\delta}/m>1$. Indeed, if we take
    $\delta=1-\eps$ we get $a^{\delta}=a^{1-\eps}$ which is greater
    than $m$ provided $\eps$ is small enough (depening on $a,m$). Hence the same
    is true for some $\delta<1-\eps$, close to $1-\eps$ (and of course we can
    assume that $\delta>\eps$).

    Integrating the former ODE on $(r+\delta,r+1-\eps)$, where $r>r_1$, yields,
    using $\gamma$,
    \begin{equation*}\label{ineq:nonex.3}
     \begin{aligned}
     \psi(r+1-\eps)  & \geq
     c\beta e^{-\alpha(r+1-\eps)}\int_{r+\delta}^{r+1-\eps}
     \exp\big(\frac{a^{s-\eps}}{m}\big)\,
     e^{\alpha s}\d s\\
    & \geq c \beta (1-\eps -\delta)e^{\alpha(\delta-1+\eps)}\exp\big(
    a^{r-\eps}\frac{a^{\delta}}{m}\Big)
    \geq (c \beta) Ce^{ \gamma a^{r-\eps}},
    \end{aligned}
    \end{equation*}
    where the constant $C>0$ depends on $\delta,\eps,\alpha,a,m$.
    With this new estimate for $\psi$, we can improve \eqref{eq:ode} as
    \begin{equation}
        \forall r>r_1\,,\quad
        (\alpha\psi(r)+\psi'(r))^m\geq (c \beta) C\exp(\gamma a^{r-\eps})\,,
    \end{equation}
    and by direct induction we obtain that for any $n\geq2$ and $r>r_1$,
    $$\psi(r+1-\eps)\geq (c \beta)^{n}C(\delta,\eps,\alpha,a,m)^ne^{\gamma^n a^r}.$$
    Since $\gamma>1$, we conclude by sending $n$ to infinity, which yields a
    contradiction since it would mean that $\psi(r)=\infty$ for $r>r_1$.

    The proof of $(i)$ is exactly the same, but it requires a
    first iteration using the growth of $f$ instead of the
    term $\beta\psi(r+1-\eps)$ in~\eqref{ineq:nonex.2}: since
    $$
    \forall r>r_1\,,\quad
    \Big[(e^{\alpha r}\psi)'e^{-\alpha r}\Big]^m\geq ce^{a^r}\,,
    $$
    we obtain for $r>r_1$
    $$
    \psi(r+1-\eps)\geq
    ce^{-\alpha(r+1-\eps)}\int_{r_1}^{r+1-\eps}e^{(a^s)/m}e^{\alpha s}\d s
    \geq cC\exp\big(\gamma a^{r}\big)\,.
    $$
    With this estimate, we are back to $(i)$ and get the contradiction by the
    same iteration process, using the $\beta\psi$-term.
\end{proof}

\begin{corollary}\label{cor:non.existence} Assume that $f\in\Rinf$
    satisfies \hyp{H0}.\\[2mm]
    \noindent $(i)$ there is no supersolution $u\in\C^1\cap\Rinf$ of \ep{}
    if $f(x)\geq C\exp(a^{|x|})$ for $|x|\gg1$ with $C>0$ and $a>m$.\\[2mm]
    \noindent $(ii)$ there is no supersolution $u\in\C^1\cap\Rinf$ of \ep{} such
    that for $|x|\gg1$, $u(x)\geq C\exp\big(a^{|x|-1}\big)$, with $C>0$ and $a>m$.
\end{corollary}

\begin{proof}
    Set $\bar f:=f-\lambda$ and we use Lemma~\ref{lem:est.L.psi} with  $\alpha=1$,
    $\beta=c_\epsilon$. We get
    $$
        \alpha\psi(r)-\beta\psi(r+1-\eps)+(\psi'(r))^m\geq
        \bar f(r)\text{ for }r\gg1\,.
    $$
    So, part $(ii)$ follows directly from applying Lemma~\ref{lem:non.existence} with
    $\bar f(r)$. For part $(i)$, just notice that for $r\gg1$,
    $\bar f(r)=f(r)-\lambda\geq (C/2)\exp(a^{r})$.

\end{proof}

Of course, this result is restrictive: there may exist some non-radial (or non
radially increasing) solution, which seems though quite improbable if $f\in\Rinf$.
But at least our result is a good hint that a more general non-existence statement
should hold.

\section{Approximate problems in bounded domains}
\label{sect:approximate}

In this section we settle and solve some approximate problems that are defined in
$B_R$. Solving such problems is quite standard for local equations (see~\cite{BarlesMeireles2017, Ichihara}), but we have to adapt here several steps to deal
with the non-local term.
In particular, we will use the Perron's method, following the standard construction; however, we do not want to skip it, since, because of the
non-local operator which is involved, we have to check and
adapt every step carefully.
Those approximate problems will be the key  in order to construct solutions of \ep{} in the
whole space, Section~\ref{sect:existence}.

First of all we adapt the definition of the non-local term to the  bounded domain~$B_R$. When solving the Dirichlet problem we need to consider the usual boundary
condition (function $g\in \C^{0,\gamma}(\partial B_R)$ for any
  $\gamma\in(0,1)$ see~\eqref{eq.approx.R.varepsilon} below), but also an \textit{outer
    condition} $\psi$, which enters in the non-local
operator, see~\cite{ChasseigneChavesRossi2006, CortazarElguetaRossi2009}.
Thus for $R>1$ we define
\begin{equation}
  \label{eq.nonlocal.Dirichlet}
 \mathcal{L}_R^\psi[v](x):=\int_{B_R} J(x-y)v(y)\d y+\int_{B_R^C} J(x-y)\psi(y)\d y-v(x).
\end{equation}
It is clear that, if $v$ is defined in the whole space, then $\mathcal{L}_R^v[v]=\mathcal{L}[v]$. Moreover, since $J$ is compactly supported on $B_1$, for $x\in B_R$, the outer term in~\eqref{eq.nonlocal.Dirichlet} becomes
  $$
   \int_{B_R^C} J(x-y)\psi(y)\d y=\int_{B_{R+1}\setminus B_R} J(x-y)\psi(y)\d y,
  $$
so that the function $\psi$ needs only to be, say, continuous on $B_{R+1}\setminus
B_R$ for $\mathcal{L}_R^\psi$ to be defined.

Now, for $\varepsilon>0$ and $R>1$ fixed, we consider
the approximate problem
  \begin{equation}
  \label{eq.approx.R.varepsilon}
   \left\{
   \begin{array}{ll}
    \lambda-\varepsilon\Delta v-\mathcal{L}_R^\psi[v]+|Dv|^m=f,\quad& x\in B_R,\\
    v=g, & x\in \partial B_R.
  \end{array}
  \right.
\end{equation}

The fundamental existence result is the following:
\begin{proposition}
\label{proposition.existence.vR.R.varepsilon}
  Let $\eps>0$, $R>1$, $f\in \C^1(B_R)$, $\psi\in\C^0(B_{R+1}\setminus
  B_R)$ and $g\in\C^{0,\gamma}(\partial B_R)$. If there exists a subsolution
  $\underline{v}\in\C^2(B_{R})\cap\C^0(\overline{B_R})$
      of~\eqref{eq.approx.R.varepsilon}, then there
  exists a solution, ${v}\in\C^{2}(B_R)\cap\C^0(\overline{B_{R}})$
  of~\eqref{eq.approx.R.varepsilon}.
\end{proposition}

\begin{remark}
  We opted to state this result assuming $f\in \C^1$, since this is the general
  assumption we make (see Section~\ref{sect:ass.pre}). But the same proof holds
  with a regularizing argument if $f$ is only continuous or  even
  $f\in\W(\mathbb{R}^N)$.
\end{remark}

In order to use Perron's method, we
have to provide, as a first step, a supersolution to problem~\eqref{eq.approx.R.varepsilon}. To this aim, let us
consider the linearized problem
\begin{equation}
  \label{eq.pbm.laplace.eps.ball}
   \left\{
   \begin{array}{ll}
        -\varepsilon\Delta\phi-\mathcal{L}_R^\psi[\phi]=M,\quad&  x\in B_R,\\
        \phi=g,& x\in \partial B_R.
 \end{array}
  \right.
  \end{equation}
  Existence and uniqueness of a solution
  $\phi\in \C^{2,\gamma}(B_R)\cap\C^{0}(\overline{B_R})$ for any
  $\gamma\in(0,1)$
  for problem~\eqref{eq.pbm.laplace.eps.ball} is obtained through a variant
  of~\cite[Theorem 6.8]{GT} that includes the non-local operator. We detail in
  the Appendix the construction and adaptations, see
  Lemma~\ref{lemma.existence.super.classical}.

\begin{lemma} \label{lemma.supersolution.continuity}
   Let $M> \|f\|_{\L^\infty(B_R)}+|\lambda|$,
   $\overline{v}\in\C^{2,\gamma}(B_R)\cap\C^0(\overline{B_R})$ be the solution
   of~\eqref{eq.pbm.laplace.eps.ball} and $\underline{v}$ a subsolution of
   \eqref{eq.approx.R.varepsilon}.
Then $\underline{v}\leq\overline{v}$ in $B_R$.
\end{lemma}

\begin{proof}

  Due to the choice of $M$, it is straightforward that  $\overline{v}$
  is a strict supersolution of~\eqref{eq.approx.R.varepsilon}. {Moreover,
  $\overline{v}=\underline{v}$ on $\partial B_R$.}

To get the comparison result, we define $w:=\underline{v}-\overline{v}\in\C^2(B_R)\cap\C^0(\overline{B_R})$, which verifies
    $$-\varepsilon\Delta w-\mathcal{L}_R^0[w]<0\quad\text{in}\  B_R\quad \mbox{and}\quad  w=0\quad\text{on}\ \partial B_R.$$
        Notice that the exterior term involving $\psi$ cancels after
    substracting the equations, hence we obtain a zero-Dirichlet problem for
    $w$, $i.e.$ with $g=\psi=0$.

    Let $x_0$ be a maximum point  of $w$ in $\overline{B_R}$. If
    $w(x_0)\leq0$ the result follows immediately, so let us assume that
    $w(x_0)>0$. In this case, notice that $x_0\in\partial B_R$ is impossible
    since $w=0$ on the boundary.

    Since $x_0\in B_R$ and $w$ is a $\C^2$-smooth function,  we have $\Delta w(x_0)\leq 0$ and the equation yields $$-\mathcal{L}^0_R[w](x_0)<0.$$
    But since $x_0$ is a point where $w$ reaches its positive maximum,
    we get
    $$
    0<w(x_0)\Big(\int_{B_R}J(x_0-y)\d y-1\Big),
    $$
    which is a contradiction since $w(x_0)>0$ and $\int_{B_R}J(x_0-y)\dy\leq1$.
\end{proof}

In the following, we introduce the critical
exponent  $\alpha_*:=(m-2)/(m-1)\in(0,1)$ as found in \cite{CapuzzoLeoniPorretta},
which is where $m>2$ is needed.

\begin{proof}[Proof of
    Proposition~{\rm\ref{proposition.existence.vR.R.varepsilon}}]
    Let
    {${\overline{v}}\in\C^2(B_R)\cap\C^1(\overline{B_R})$} be a supersolution
    of~\eqref{eq.approx.R.varepsilon} and consider the set \begin{equation*}
    \mathcal{S}:=\{v_S\in \C^{0,\alpha_*}(\overline{B_R}):
        v_S \textrm{{viscosity subsolution}
            of~\eqref{eq.approx.R.varepsilon} with }
      \underline{v}\leq v_S\leq \overline{v} {\rm\  in\ } \overline{B_R}\}.
\end{equation*}
    The set is non-empty since $\underline{v}\in \mathcal{S}$.
    For any $x\in B_R$, we set
    $$v(x):=\sup_{v_S\in\mathcal{S}}v_S(x)$$
    which is well-defined,  since all the functions $v_S$ in $\mathcal{S}$
    are bounded above by $\overline{v}$.
    We first notice that $\underline{v}\leq v\leq\overline{v}$ in $B_R$
    and that necessarily $v=\underline{v}=\overline{v}$ on $\partial B_R$.

    Concerning the regularity of $v$, we use the
    estimates of \cite[Theorem 1.1]{CapuzzoLeoniPorretta}. Writing~\eqref{eq.approx.R.varepsilon} as
    $$
    v_S-\varepsilon\Delta v_S+|Dv_S|^m\leq F(x),
    $$
    where
    $$
    F(x):=\int_{B_R} v_S(y)J(x-y)\d y+\int_{B_{R+1}\setminus B_R}
    \psi(y)J(x-y)\d y+f(x)-\lambda,
    $$
    we see that there exists a constant $K$ depending only on $\|(v_S)^-\|_{\L^\infty(B_R)}$ and
    $\|F\|_{\L^\infty(B_R)}$ such that
    $$
    |v_S(x)-v_S(y)|\leq K|x-y|^{\alpha_*}\quad \text{in}\quad \overline{B_{R}}.
    $$
        Now, for any $v_S\in\mathcal{S}$ we have $(v_S)^-\leq(\underline{v})^-$ and
    $v_S\leq\overline{v}$. So, both functions $F(x)$ and $(v_S)^-(x)$ are uniformly
    bounded in $B_R$ with respect to $v_S\in\mathcal{S}$.
    We deduce that the subsolutions $v_S$ in $\mathcal{S}$ are uniformly
    H\"older continuous up to the boundary, which implies that
    $v\in\C^{0,\alpha_*}(\overline{B_R})$.

    \noindent{\sc Claim --}  $v$ is a viscosity solution
    of~\eqref{eq.approx.R.varepsilon}.
    It is standard that, being defined as a supremum of subsolutions, $v$ is also
    a  subsolution of~\eqref{eq.approx.R.varepsilon}. Hence $v\in\mathcal{S}$
    and it remains to prove that $v$ is a supersolution
    of~\eqref{eq.approx.R.varepsilon}. Since by construction
    $v=g$ on $\partial B_R$, we only need to check the supersolution condition
    inside $B_R$, which is done as usual through the construction of a bump
    function.

    We proceed by contradiction: let us assume that $v$ is not a supersolution of~\eqref{eq.approx.R.varepsilon}.
    Then, there exists a fixed $\bar x\in B_R$ and $\varphi\in\C^2(B_R)$ such
    that $v-\varphi$ has a
    minimum at $\bar x$ and
        \begin{equation}\label{eq:perron.contr} -\varepsilon\Delta \varphi (\bar
            x)-\mathcal{L}_R^\psi[v](\bar x)+|D\varphi(\bar
            x)|^m-f(\bar x)+\lambda<0.
        \end{equation}
    We can assume with no restriction that $v(\bar x)=\varphi(\bar x)$, hence
    $v\geq\varphi$ in $B_R$.

    Moreover, we claim that $v(\bar x)<\overline{v}(\bar x)$.
        Indeed, assuming otherwise that $v(\bar x)=\overline{v}(\bar x)$,
    since $v\leq\overline{v}$ we would have $v(\bar x)=\varphi(\bar
    x)=\overline{v}(\bar x)$ and $\varphi\leq v\leq\overline{v}$,
    which together imply that $\varphi-\overline{v}$
    has a minimum at $\bar x$.
    Consequently
    $D\varphi(\bar x)=D\overline{v}(\bar x)$ and $\Delta \varphi(\bar
    x)-\Delta\overline{v}(\bar x)\geq 0$. Replacing $\varphi$ by $\overline{v}$
    in~\eqref{eq:perron.contr} at the point $\bar x$ we would get
    \begin{equation*} -\varepsilon\Delta \overline{v}(\bar x)
        -\mathcal{L}_R^\psi[v](\bar x)+|D\overline{v}(\bar
        x)|^m-f(\bar x)+\lambda<0.
    \end{equation*}
    But since $v(\bar x)=\overline{v}(\bar x)$ and
    $v\leq\overline{v}$ in $B_R$, we have
    $\mathcal{L}_R^\psi[v](\bar
    x)\leq\mathcal{L}_R^\psi[\overline{v}](\bar x)$ and we see that
    \begin{equation*} -\varepsilon\Delta\overline{v}(\bar
        x)-\mathcal{L}_R^\psi[\overline{v}](\bar
        x)+|D\overline{v}(\bar x)|^m-f(\bar x)+\lambda<0.
    \end{equation*}

    This contradicts the fact that $\overline{v}$ is a supersolution
    of~\eqref{eq.approx.R.varepsilon}.  Hence $v(\bar x)< \overline{v}(\bar x)$
    and we define, for any $y\in B_R$, the bump function
    \begin{equation*}
     \label{eq.definition.vdelta} v_\delta(y):=\max\{
            v(y);\varphi(y)+\delta-|\bar x-y|^2\}.  \end{equation*}
    Notice that by construction,
    \begin{equation}\label{est:bump}
    \begin{cases}v_\delta(y)=v(y) & \text{if } y\notin
        B_{\delta^{1/2}}(\bar x),\\
        v(y)\leq v_\delta(y)\leq v(y)+\delta\qquad & \text{if }
    y\in B_{\delta^{1/2}}(\bar x).
    \end{cases}
    \end{equation}
     The strategy is to prove that for $\delta>0$ small enough,
     $v_\delta\in\mathcal{S}$, which contradicts the definition of $v$ as a sup,
     since $v_\delta(\bar x)=\varphi(\bar x)+\delta>v(\bar x)$. We divide this
     into four steps.

   \noindent$(i)$ {\sc Regularity --} The function $v_\delta$, defined as a
   maximum of two functions which belong to $\C^{0,\alpha_*}(\overline{B_R})$,
   belongs itself to $\C^{0,\alpha_*}(\overline{B_R})$.

    \noindent$(ii)$ {\sc Bounds --}
    It is clear by construction that $v_\delta\geq \underline{v}$ in $B_R$.

    On the other hand,
    if $|y-\bar x|^2\geq\delta$ then $v_\delta(y)=v(y)$, see~\eqref{est:bump}.
    Hence, outside $B_{\delta^{1/2}}(\bar x)$, we have $v_\delta\leq
    \overline{v}$.   Moreover, since $v(\bar x)<\overline{v}(\bar x)$, for
    $\delta$ small, we have $v_\delta\leq\overline{v}$ for $y\in
    B_{\delta^{1/2}}(\bar x)$. Therefore $v_\delta\leq \overline{v}$ in $B_R$
    for $\delta$ small enough.

   \noindent$(iii)$ {\sc Subsolution condition --}
    As before, outside $B_{\delta^{1/2}}(\bar x)$, we have that $v_\delta=v$.
    Hence $v_\delta$ is a subsolution in $B_{\delta^{1/2}}^C(\bar x)$.

    Let $y\in B_{\delta^{1/2}}(\bar x)$ and $\varrho\in\C^2(B_R)$ be a test function
    such that $v_\delta-\varrho$ has a (strict) maximum zero at $y$.
    We have to prove that in this situation $v_\delta$ verifies
    \begin{equation*}
    \label{eq.weps.subsolution}
        -\varepsilon\Delta\varrho(y) -\mathcal{L}_R^\psi[v_\delta](y)+
        |D\varrho(y)|^m-f(y)+\lambda\leq 0.
    \end{equation*}
    Since $y\in B_{\delta^{1/2}}(\bar x)$ we have
    $\varrho(y)=\varrho(\bar x)+o_\delta(1)$ uniformly with respect to $y\in
    B_{\delta^{1/2}}(\bar x)$.   Similarly, $
    |D\varrho(y)|^m=|D\varrho(\bar x)|^m+o_\delta(1)$ and
    $\varepsilon\Delta\varrho(y)=\varepsilon\Delta\varrho(\bar x)+o_\delta(1)$.

    Moreover, $-\mathcal{L}_R^\psi[v_\delta](y) =
    -\mathcal{L}_R^\psi[v_\delta](\bar x) + o_\delta(1)$, and the
    facts that $v_\delta \geq v$ and
    $v_\delta(\bar x)=v(\bar x)+\delta$ imply that
    $-\mathcal{L}_R^\psi[v_\delta](\bar x)
    \leq -\mathcal{L}_R^\psi[v](\bar x)+\delta$.
      Finally, since $f$ is continuous,
    $f(y)=f(\bar x)+o_\delta(1)$.  Gathering these estimates
    we obtain that for any $y\in B_{\delta^{1/2}}(\bar x)$,
    $$
    \begin{aligned}
      -\varepsilon\Delta\varrho(y) &-\mathcal{L}_R^\psi[v_\delta](y)+
        |D\varrho(y)|^m-f(y)+\lambda\\
        \leq &-\varepsilon\Delta\varrho(\bar x)-
        \mathcal{L}_R^\psi[v](\bar x)+
         |D\varrho(\bar x)|^m-f(\bar x)+\lambda +o_\delta(1).
    \end{aligned}
    $$

    Finally, since $\varrho(\bar x)=v_\delta(\bar x)+o_\delta(1)=\varphi(\bar
    x)+o_\delta(1)$,  we deduce that for $\delta>0$ small enough, according
    to~\eqref{eq:perron.contr}, the bump function $v_\delta$ is a subsolution of
    \eqref{eq.approx.R.varepsilon} in $B_{\delta^{1/2}}(\bar x)$. Therefore,
    $v_\delta=\max(v,\varphi_\delta)$ is a subsolution in $B_R$.

    \noindent $(iv)$ {\sc Contradiction --}
    The above points $(i)-(ii)-(iii)$ imply that $v_\delta\in\mathcal{S}$,
    which is a contradiction with the definition of $v$, since
    $v_\delta(\bar x)>v(\bar x)$.

    We conclude that $v\in \C^{0,\alpha_*}(\overline{B_R})$ is a supersolution
    of~\eqref{eq.approx.R.varepsilon} in $B_R$ and since it is also a
    subsolution, it is a (viscosity) solution of~\eqref{eq.approx.R.varepsilon}.

     Finally, to get that $v\in\C^2(B_R)$ we use some standard bootstrap regularity
    estimates. Notice first that, since  $v$ and $\psi$ are continuous, both
    integrals terms in $\mathcal{L}^\psi_R[v]$ are at least Lipschitz
    continuous. Since, by assumption, $f$ is also Lipschitz,  we apply
    \cite[Theorem 3.1]{CapuzzoLeoniPorretta}, which
    implies that $v$ is at least Lipschitz continuous, locally in $B_R$.
    Thus, the Lipschitz function $v$ satisfies an equation of the form
    $-\varepsilon\Delta v = \tilde{F}\in\L^\infty_{\rm loc}(B_R)$ in the
    viscosity sense.

    This implies that $v$ is actually also a \textit{weak} solution of this
    equation. This is a quite straightforward and standard statement in viscosity solutions' theory
    which comes from the fact that we can regularize $v$ as $v_n$ and pass to
    the limit in the weak sense in the equation.

    By standard regularity results, it follows that
    $v\in \mathrm{W}^{2,p}_{\mathrm{loc}}(B_R)$ for any $p>1$ so that
    $v\in\C^{1,\alpha}$ for any $\alpha\in(0,1)$. Hence
    $\tilde{F}$ is in fact in $\C^{0,\alpha}(B_R)$
    and from this we deduce that $v\in\C^{2,\alpha}(B_R)$ for any
    $\alpha\in(0,1)$.
\end{proof}

We end this section by introducing a uniform (in $\varepsilon$)
estimate of the solution $v$:

\begin{lemma}\label{lem:est.vbar}
    For any $R>1$ there exists a constant $C=C(R,\psi,g,f)>0$ such that for
    $\varepsilon>0$ small enough,
    $\|v\|_{\L^\infty(B_R)}\leq C(R,\psi,g,f)$.
\end{lemma}

\begin{proof}
   Consider the following equation
   \begin{equation}
     \label{eq.JJ}
     -\mathcal{L}_R^{\psi}[\chi]=M+1\quad\text{in}\quad B_R,
   \end{equation}
   where $M$ is defined in Lemma~\ref{lemma.supersolution.continuity}.
   We refer to~\cite[Appendix]{MelianRossi} for existence of a solution $\chi\in
   L^1(B_R)$ of~\eqref{eq.JJ}. Notice that, since both integrals in the non-local term are at least continuous,
   we have $\chi\in\C^0(\overline{B_R})$. Then we consider a resolution of the
   identity $(\rho_k)_{k\in\N}$ and set
   $$\bar\chi_k:=\rho_k\ast\big(\chi+c\big),\quad c:=\|g\|_{\L^\infty(B_R)}
   + \|\chi\|_{\L^\infty(B_R)},$$
   so that for any $k\in\N$, $\bar\chi_k\in\C^2(B_R)\cap\C^0(\overline{B_R})$
   and $\bar\chi_k\geq g$ on $\partial B_R$.
   Since $\bar\chi_k\to\chi+c$ uniformly in $B_R$ and $c>0$, it follows that
   $$-\mathcal{L}^\psi_R[\bar\chi_k]=-\mathcal{L}^\psi_R[\chi]-
   \mathcal{L}^0_R[c]+o_k(1)\geq M+1+o_k(1),$$
   where $o_k(1)$ vanishes as
$k\to\infty$, uniformly with respect to $x\in B_R$.
  We first choose $k=k_0$ big enough (but fixed) so that the right-hand side
   above is greater than $M+1/2$. Then, it follows that for
   $0<\varepsilon<\varepsilon_0(k_0)$ we have $\varepsilon\|\Delta
   \bar\chi_{k_0}\|_{\L^\infty(B_R)} < 1/2$. Hence, setting
   $\omega(x):=\bar\chi_{k_0}(x)$ we find that for $\varepsilon$ small enough,
   $$-\varepsilon\Delta \omega -\mathcal{L}^\psi_R[\omega]>M$$
   which means that $\omega$ is a supersolution
   of~\eqref{eq.pbm.laplace.eps.ball} such that $\omega\geq g$ on $\partial
   B_R$. By the comparison principle (which can be proved exactly as in the proof of
   Lemma~\ref{lemma.supersolution.continuity} and Theorem~\ref{thm:comparison.principle}) we get that
   $\overline{v}\leq\omega \leq \|\omega\|_{\L^\infty(B_R)}:=C(R,\psi,g,f)$ in $B_R$. And the
   result follows since by construction $v\leq \overline{v}$.
\end{proof}

\begin{remark}\label{rem:unif.n}
    We will use later (see Lemmas~\ref{lem:murne} and \ref{lem.estimates.0})
    that the above  estimate is uniform with respect to the
    data $\psi,g,f$ provided they remain bounded. Especially,
    this holds true if we take approximations $\psi_n,g_n,f_n$ that converge
    uniformly in $B_R$.

\end{remark}

\section{Existence results for \ep{}}
\label{sect:existence}

The aim of this section is to prove that there exists at least one value of
$\lambda$ for which problem \ep{} is solvable. Moreover, the solution turns out to be bounded by the supersolution $\Psi_\lambda$  constructed in Section~\ref{sect:ass.pre}, see~\eqref{eq:super.lambda}.

\begin{theorem}
\label{thm.existence}
    If there exists a strict viscosity subsolution,
    $\underline{u}\in\Wloc(\mathbb{R}^N)$, of $\ep{\lambda}$ such that
    $\underline{u}\leq \Psi_\lambda$ in $\R^N$ and
    $\underline{u}(0)=0$, then there exists  a viscosity solution,
    ${u_\lambda}\in\Wloc(\mathbb{R}^N)$,  of $\ep{\lambda}$,  such that
    $u_{\lambda}\leq\Psi_{\lambda}$ in $\R^N$
     and $u_\lambda(0)=0$.
\end{theorem}

We reduce the proof of this result  to solve the approximate problem defined on~$B_R$,
see~\eqref{eq.approx.R.varepsilon}, and then pass to the limit first as
$\varepsilon$ tends to zero, and second as $R$ tends to $+\infty$. To perform
this we need two main ingredients: $(i)$ smooth the data (the right-hand side $f$ and
the subsolution $\underline{u}$); $(ii)$ get some local uniform bounds independent
of $\varepsilon$ and $R$ to pass to the limit.

Let $\rho_n$ be a resolution of the identity and set $f_n:=\rho_n\ast f$,
$\psi_n:=\rho_n\ast \underline{u}$. Then both $f_n$ and $\psi_n$ are smooth and
converge uniformly in $B_R$ to $f$ and $\underline{u}$ respectively.
In the following, $o_n(1)$ stands for any quantity which vanishes as
$n\to\infty$, uniformly with respect to $x\in B_R$. Thus,
$\psi_n=\underline{u}+o_n(1)$ and $f_n=f+o_n(1)$.

Observe that, as we mentioned in Remark~\ref{rem:regularity.f}, our general
assumption on $f$ is~$\C^1$. This will be not enough here, since at some steps
in this section we will have to compute $\Delta\Psi_\lambda$. This is  why we
have to do an approximation argument and use~$f_n$.

We consider~\eqref{eq.approx.R.varepsilon}$_n$; i.e, problem~\eqref{eq.approx.R.varepsilon} with data $f=f_n$, outer condition $\psi=\psi_n$,
and boundary data $g=\psi_n$ on $\partial B_R$.

The first result we prove states that we can use $\psi_n$ as a subsolution to~\eqref{eq.approx.R.varepsilon}$_n$ in
$B_R$. {Notice that during the proof we will choose $\varepsilon<\eps_0(n)$.
This will not be a problem when passing to the limit, since we will first send
$\varepsilon$ to zero (Lemma~\ref{lemma:epsilon.to.0}) and then $n$ to
infinity.}
\begin{lemma}\label{lem:approx.rhon}
    There exists $\eta>0$ such that {for $n$ big enough and
        $0<\varepsilon<\eps_0(n)$,} the smooth function $\psi_n$ satisfies
    \begin{equation}
      \label{eq:psin.is.a.subsolution}
      \lambda-\varepsilon\Delta\psi_n-\mathcal{L}_R^{\psi_n}[\psi_n]+
    |D\psi_n|^m\leq f_n-\eta/2 \quad\text{in}\quad B_R.
    \end{equation}
    Moreover, $\psi_n(0)=o_n(1)$ and $\psi_n\leq \Psi_\lambda+o_n(1)$
    uniformly in $B_R$.
\end{lemma}
\begin{proof}
    Since $\psi_n$ converges uniformly in $B_R$ to
    $\underline{u}$,
    it is a direct consequence of the assumptions on $\underline{u}$ that $\psi_n(0)=o_n(1)$ and $\psi_n\leq \Psi_\lambda+o_n(1)$.

    In order to prove the first part of the lemma, notice that, since
    $\underline{u}$ is a locally Lipschitz strict subsolution of~\ep{\lambda} in
    $\R^N$, there exists $\eta>0$ such that for almost any $x\in B_R$,
    \begin{equation}
      \label{eq:subsolution.inequality}
    \lambda-\mathcal{L}[\underline{u}]+
    |D\underline{u}|^m \leq f-\eta.
    \end{equation}
        We can then estimate the terms in~\eqref{eq:psin.is.a.subsolution} as
    follows: for the non-local term we notice that
    $\mathcal{L}_R^{\psi_n}[\psi_n]=\mathcal{L}[\psi_n]=
    \mathcal{L}[\underline{u}]+o_n(1)$.
    {
    To deal with the gradient term, since $\underline{u}$ is only Lipschitz
    locally, not $\C^1$ we cannot use $|D\psi_n|^m=|D\underline u|^m+o_n(1)$, but
    we use Jensen's inequality: for any convex function $\varphi:\R^N\to\R$,
    any probability measure $\sigma$ on $\R^N$ and any $\sigma$-integrable
    function $p:\R^N\in\R^N$ there holds
    $$\varphi\Big(\int p(y)\d\sigma(y)\Big)\leq\int
    (\varphi\circ p)(y)\d\sigma(y).$$
    We applied it to $\varphi(s)=|s|^m$ which is convex since $m>1$,
    $\sigma(y)=\rho_n(x-y)$ where $x$ is fixed and $p(y)=D\underline{u}(y)$
    which is $\sigma$-integrable, because it is locally bounded, while $\sigma$ is
    compactly supported. We deduce that for any $x$,
    $$|D\psi_n|^m(x)=|\rho_n\ast D\underline{u}|^m(x)\leq
    \rho_n\ast\big(|D\underline{u}|^m\big)(x)\,.
    $$
    Using the subsolution inequality~\eqref{eq:subsolution.inequality} for $\underline{u}$ we get
    $$|D\psi_n|^m\leq \rho_n\ast(f-\eta-\lambda+\mathcal{L}[\underline{u}])=
    f_n-\eta-\lambda+\mathcal{L}_R^{\psi_n}[\psi_n]+o_n(1)\,,$$
    where the $o_n(1)$ is uniform with respect to $x\in B_R$.
    Moreover, $-\varepsilon\Delta\psi_n$ is as
    small as we want if we choose $\varepsilon<\eps_0(n)$, $n$ being fixed.
    }
    Hence,
    we deduce that
    $$\lambda-\varepsilon\Delta\psi_n-\mathcal{L}_R^{\psi_n}[\psi_n]+
    |D\psi_n|^m \leq f_n-\eta+o_n(1)+\varepsilon\|\Delta\psi_n\|_{\L^\infty(B_R)}
    \quad\text{in}\quad B_R\,,$$
    which yields~\eqref{eq:psin.is.a.subsolution}, provided
    $\eps <\eps_0(n)$.
\end{proof}

\newcommand{\vrne}{v_{R,n,\varepsilon}}
\newcommand{\wrne}{w_{R,n,\varepsilon}}
\newcommand{\murne}{\mu_{R,n,\varepsilon}}
\newcommand{\Frne}{F_{R,n,\varepsilon}}

A supersolution to~\eqref{eq.approx.R.varepsilon}$_n$ is obtained as in Section~\ref{sect:approximate}, but now
using~\eqref{eq.pbm.laplace.eps.ball}$_n$, which is~\eqref{eq.pbm.laplace.eps.ball} with $\psi$ replaced by $\psi_n$ and $g=\psi_n$. Hence, using  Proposition~\ref{proposition.existence.vR.R.varepsilon}, we find a  solution
of~\eqref{eq.approx.R.varepsilon}$_n$, defined for $\eps <\eps_0(n)$, that we denote $\vrne$. By construction we have $\vrne\geq\psi_n=\underline u+o_n(1)$.
We define also $$\wrne(x):=\vrne(x)-\vrne(0),$$ so that
$\wrne(0)=0$ and $\wrne(x)=\psi_n(x)-\vrne(0)$ on $\partial B_R$.
Moreover, $\wrne$ satisfies
\begin{equation}
  \label{eq.w}
     \lambda-\varepsilon\Delta \wrne-
     \mathcal{L}_R^{\psi_n}[\wrne]+|D\wrne|^m=
     f_n+\murne,\quad x\in B_R,
 \end{equation}
 where $\murne(x):=\mathcal{L}_R^0[\vrne(0)](x)$.
  Concerning this term (which does not exist  neither in the local case nor in the problem defined in the whole space),
 we have a first estimate which follows directly from the construction of $\vrne$ and the fact that $\underline u(0)=0$:
 \begin{equation}
   \label{eq:first.estimate.mu}
   \begin{aligned}
     \murne(x)&=\vrne(0)\Big(\int_{B_R}J(x-y)\d y-1\Big)\\
     &\geq (\underline u(0)+o_n(1))\Big(\int_{B_R}J(x-y)\d y-1\Big)=o_n(1).
    \end{aligned}
 \end{equation}
 Moreover,
 introducing the indicator function $\ind{A}$ of $A$, we have:
 \begin{lemma}\label{lem:murne}
     For any $R>1$, there exists a constant $\nu(R)>0$ such that, for $n$ big enough and
        $0<\varepsilon<\eps_0(n)$,
 \begin{equation*}\label{mu.rne}
     |\murne(x)|\leq \nu(R)\cdot \ind{B_{R}\setminus B_{R-1}}(x).
 \end{equation*}

 \end{lemma}
\begin{proof}
  Notice that, for any $|x|\leq R-1$, since $J$ is compactly supported in $B_1$,
  $$\murne(x)=\vrne(0)\Big(\int_{B_R}J(x-y)\d y-1\Big)=0.$$
  Hence, we first deduce that $\murne$ is compactly supported in $B_R\setminus
  B_{R-1}$. Second, Lemma~\ref{lem:est.vbar} gives an estimate of $\vrne(0)$
  by a constant which is independent of $\varepsilon$ (small enough). Actually,
  this estimate can be found uniform in $n$, since $\psi_n$ and $f_n$ converge
  uniformly in $B_R$ to $\underline{u}$ and $f$ respectively.
  \end{proof}

  We now prove a local bound for
 $\wrne$, independent of $\eps,n$ and $R$,
 in terms of the supersolution $\Psi_\lambda$.

\begin{lemma}
\label{lemma:psiboundsw}
    For any $R>1$ fixed, $\wrne \leq \Psi_{\lambda}+o_n(1)$
    in $\overline B_R$.
\end{lemma}

\begin{proof} By Lemma~\ref{lem:approx.rhon} we have that $\psi_n\leq \Psi_n+o_n(1)$
   uniformly in $B_R$. Hence
  $$\mathcal{L}_R^{\psi_n}[\Psi_{\lambda}]\leq\mathcal{L}_R^{\Psi_\lambda+o_n(1)}
  [\Psi_\lambda]\leq \mathcal{L}[\Psi_{\lambda}]+o_n(1).$$

 Since  $\Psi_\lambda$ is a strict supersolution of $\epl$ in $\mathbb{R}^N\setminus\{0\}$, we
  deduce that for $\varepsilon\ll 1$ and $n$ big enough
  \begin{equation}
   \label{eq.supersolution.condition.psi}
   \lambda-\varepsilon\Delta \Psi_{\lambda}-\mathcal{L}_R^{\psi_n}
   [\Psi_{\lambda}]+|D\Psi_{\lambda}|^m>f_n+o_n(1), \quad {\rm for\quad}
    |x|\neq 0.
  \end{equation}

  Let $x_0\in\overline B_R$ be such that $(\wrne-\Psi_{\lambda})(x_0)\geq
(\wrne-\Psi_{\lambda})(x)$ for all $x\in\overline B_R$.

If  $x_0\in B_R\setminus\{0\}$ and
$(\wrne-\Psi_{\lambda})(x_0)\leq0$, the result follows. On the other
hand, if $x_0\in B_R\setminus\{0\}$ is a point where
$\wrne-\Psi_{\lambda}$ achieves a positive maximum, then at this point
$D\wrne(x_0)=D\Psi_{\lambda}(x_0)$, $\Delta \wrne(x_0)\leq \Delta\Psi_{\lambda}(x_0)$ and
$\mathcal{L}_R^{\psi_n}[\wrne](x_0)\leq
\mathcal{L}_R^{\psi_n}[\Psi_{\lambda}](x_0)$.
Using that $\wrne$ satisfies~\eqref{eq.w} together with~\eqref{eq:first.estimate.mu} we get a
contradiction with~\eqref{eq.supersolution.condition.psi}.

Hence, either $x_0=0$ or $x_0\in\partial B_R$. In both cases we get
$\wrne\leq \Psi_{\lambda}+o_n(1)$ in $\overline B_R$ as follows:

\noindent$(i)$ If $x_0\in \partial B_R$ then, since by construction
$\vrne(x_0)=\psi_n(x_0)\leq \Psi_\lambda(x_0)+o_n(1)$
and $\vrne(0)\geq \psi_n(0)=o_n(1)$,
we get $\wrne(x_0)\leq \Psi_{\lambda}(x_0)+o_n(1)$.

\noindent$(ii)$ If $x_0=0$ then $\wrne(0)= 0 \leq \Psi_{\lambda}(0)+o_n(1)$.
\end{proof}

Our next aim is letting $\eps\to 0$. To do so, we need estimates that are
independent of $\eps$ both from above and below.

\begin{lemma}\label{lem.estimates.0}
    For any $R>1$ fixed, there exists $C_1(R),C_2(R)>0$, independent of $\varepsilon$,
     such that, for $n$ big enough and
        $0<\varepsilon<\eps_0(n)$,  $-C_1(R)+o_n(1)\leq \wrne\leq C_2(R)+o_n(1)$ in $B_R$.
\end{lemma}

\begin{proof}
   The upper bound is a direct consequence of Lemma~\ref{lemma:psiboundsw}.

   For the lower bound we use that, by construction, $\vrne\leq \overline{v}$. This implies, see the proof of
   Lemma~\ref{lem:est.vbar}, that $\vrne\leq C(R)$ in $B_R$. Here, as in
   Lemma~\ref{lem:murne}, we notice that the estimate is uniform with respect to
   $n$, since $\psi_n$ and $f_n$ converge uniformly in $B_R$.
   Now, using this bound,
  $\wrne(x)=\vrne(x)-\vrne(0)\geq
   \psi_n(x)-C(R)=\underline{u}-C(R)+o_n(1)\geq -C_1(R)+o_n(1)$ in $B_R$.
\end{proof}

\begin{lemma}
\label{lemma:epsilon.to.0}
  For any $R>1$ fixed,  the sequence  of solutions (up to a subsequence)
  $\{\wrne\}_\varepsilon$ of~\eqref{eq.w} converges locally uniformly
  in $B_R$ as $\varepsilon\to 0$ to a continuous viscosity solution
  $w_{R,n}$ of
  \begin{equation}
    \label{eq.w.without.eps}
    \lambda-\mathcal{L}_R^{\psi_n}[w_{R,n}]+|Dw_{R,n}|^m=f_n+
    \mu_{R,n}, \quad x\in B_R,
 \end{equation}
 where $\mu_{R,n}$ is compactly supported in $B_R\setminus B_{R-1}$.
 Moreover, $w_{R,n}(0)=0$ and $w_{R,n}\leq\Psi_{\lambda}$ in $B_R$.
\end{lemma}

\begin{proof}
    By adding and subtracting the term $\wrne$, we rewrite
    equation~\eqref{eq.w} under the form
    $$\wrne-\eps\Delta
    \wrne+|D\wrne|^m=\Frne,$$
    where $\Frne:=f_n+\murne+
    \mathcal{L}_R^{\psi_n}[\wrne]+\wrne-\lambda$.
    Using  the estimates of \cite[Theorem 3.1]{CapuzzoLeoniPorretta}
    we have that for any $R'<R$ there exists a constant $K$ depending only on
    $R',\|\Frne\|_{L^\infty(B_R)}$ and $\|(\wrne)^-\|_{L^\infty(B_R)}$ such
    that
    $$|\wrne(x)-\wrne(y)|\leq K|x-y|\quad \text{in}\quad B_{R'}.$$
    By Lemma~\ref{lem.estimates.0}, we know that $\wrne$ is bounded uniformly
    by some $C(R)$ in $B_R$, with respect to $\eps >0$ small enough (and $n$).
    Moreover, Lemma~\ref{lem:murne} provides a uniform estimate for
    $\murne$. Finally, we can estimate $\mathcal{L}_R^{\psi_n}[\wrne]$ using the
    bounds given in Lemma~\ref{lem.estimates.0}, and we get for some constant
    $C(R)>0$,
    $$\|\Frne\|_{\L^\infty(B_R)}\leq\|f_n\|_{\L^\infty(B_R)}+
    |\lambda|+C(R).$$
    From this we deduce that $K$ can be chosen independent of $\eps >0$ small
    and we obtain a local uniform bound in $\C^{0,\alpha}(B_R)$ as $\eps\to0$.
    Passing to the limit is done by Ascoli's Theorem and the stability property
    of viscosity solutions: up to an extraction, $\wrne\to w_{R,n}$ in $B_R$
    which is a viscosity solution of \eqref{eq.w.without.eps}.
\end{proof}

The last steps consist in sending $n,R\to +\infty$, for which we have to find other
local estimates, now independent of $R$ and $n$. This time we use gradient
estimates, which are provided by a sort of \textit{implicit control} of the
equation.
\begin{lemma}
\label{lem:est.autocontrol}
    Fix $R_0>0$. Then for any $n$ big enough and any $R>R_0+1$,
    there exists a constant $C=C(R_0)$ such
    that $\|w_{R,n}\|_{\W(B_{R_0})}\leq C$.
\end{lemma}

\begin{proof}
   The continuous viscosity solution $w_{R,n}$ of \eqref{eq.w.without.eps} is
   Lipschitz continuous in $B_R$. Hence it is differentiable almost everywhere
   and equation~\eqref{eq.w.without.eps} holds almost everywhere.  Moreover,
   since all the terms in~\eqref{eq.w.without.eps} are continuous, the equation
   holds everywhere in $B_R$.

   As a consequence, since $\mu_{R,n}=0$ on $B_{R_0}\subset B_{R-1}$,
   we can estimate the gradient term as follows,
   \begin{equation}
  \label{eq:autocontrol1}
  \begin{aligned}
      \sup_{B_{R_0}} |Dw_{R,n}|^m &\leq \sup_{B_{R_0}}|f_n|+
      |\lambda|+\sup_{B_{R_0}}\big|\mathcal{L}_R^{\psi_n}[w_{R,n}]\big|.
\end{aligned}
\end{equation}

On $B_{R_0+1}\setminus B_{R_0}$ we use the bound $\psi_n\leq\Psi_\lambda+o_n(1)$, see Lemma~\ref{lem:approx.rhon},
for the non-local term, so that for $n$ big enough, we have the \textit{implicit estimate}
$$
\begin{aligned}
    \sup_{B_{R_0}} |Dw_{R,n}|^m  &\leq C_0(R_0)+2\sup_{B_{R_0}}|w_{R,n}|+
 \sup_{B_{R_0+1}\setminus B_{R_0}}|\Psi_\lambda|\\
 &\leq C_1(R_0) +C_2(R_0)\sup_{B_{R_0}}|Dw_{R,n}|,
\end{aligned}
$$
where we have used the Mean Value Theorem and the fact that $w_{R,n}(0)=0$
to get the last inequality.

Setting now $X:=\sup_{B_{R_0}}|Dw_{R,n}|$ we have $X^m\leq C_2X+C_1$ for some
constants $C_1(R_0),C_2(R_0)$ independent of $R$. So, since $m>1$, there exists a
positive constant $C_3=C_3(R_0)$, depending only on $R_0$, such that
$\sup_{B_{R_0}} |Dw_{R,n}| \leq C_3(R_0)$.

Using again that $w_{R,n}(0)=0$, we deduce also that
$\|w_{R,n}\|_{\L^\infty(B_{R_0})}\leq C_4(R_0)$ for some $C_4(R_0)>0$. Gathering
these estimates, we get $\|w_{R,n}\|_{\W(B_{R_0})} \leq C(R_0)$ for some
$C(R_0)>0$, which is the desired result.
\end{proof}

We can finally complete the existence result:
\begin{proof}[Proof of Theorem~{\rm\ref{thm.existence}}]
    Since the sequence $\{w_{R,n}\}$ is locally bounded in $\W(B_R)$, independently
    of $n$, by using Ascoli's Theorem we can pass to the limit  as $n\to\infty$
     in $B_R$. We skip the details of this passage to the limit, which
    is straightforward, and yields a  solution (passing to the limit in~\eqref{eq.w.without.eps}) $w_R$, which is locally bounded in $\W(B_R)$,  of
    \begin{equation}\label{eq.lim.R}
    \lambda-\mathcal{L}_R^{\underline{u}}[w_R]+|Dw_R|^m=f+
    \mu_R, \quad x\in B_R,
    \end{equation}
    where $\mu_R$ is still supported on $B_{R}\setminus B_{R-1}$.
    Then, we send $R\to\infty$ and get that the functions $w_R$ converge
    locally uniformly to a function
    $u_{\lambda}\in \W_{\rm
        loc}(\R^N)$, such that $ u_{\lambda}\leq \Psi_{\lambda}$
    in $\R^N$ and $u_\lambda(0)=0$.

    Moreover $u_\lambda$ verifies \ep{\lambda}. Indeed, as a consequence of the
    Dominated Convergence Theorem we have
    $\mathcal{L}_R^{\underline{u}}[w_R]\to\mathcal{L}[u_{\lambda}]$ locally
    uniformly and the
    correction term $\mu_R$ vanishes locally uniformly as
    $R\to\infty$. So, we can pass to the limit in
    \eqref{eq.lim.R} in the viscosity sense to get the result.
\end{proof}

We conclude this section by the following result, which provides indeed solutions of~\ep{}
for certain values of $\lambda$.

\begin{corollary}\label{cor:existence}
  For any $\lambda\leq \min(f)$ the problem \ep{\lambda} is solvable and the
  constructed solution solution $u_\lambda$ satisfies $u_\lambda(0)=0$,
  $u_\lambda\leq \Psi_\lambda$.  \end{corollary}

\begin{proof}
  If $\lambda<\min(f)$, we can take $\underline{u}=0$ as a strict subsolution
  of~\ep{\lambda}in Theorem~\ref{thm.existence}.

      If $\lambda=\min(f)$,  $\underline{u}=0$ is not a {strict}
      subsolution of~\ep{\lambda}, but it is a {regular} subsolution. So in this case
      we do not need to use the strict subsolution property to regularize
      $\underline{u}$ (see Lemma~\ref{lem:approx.rhon}) into a smooth subsolution. Hence the above construction
      also works.
 \end{proof}

\section{Existence of a critical constant}
\label{sect:critical}

In this Section we investigate the existence of a critical ergodic constant
$\lambda_*$. Following \cite{BarlesMeireles2017, Ichihara}, it would seem
natural to consider the supremum of all
$\lambda$'s such that there exist a solution (or a
subsolution) $u$ of \ref{eq:EP}:

$$\lambda_\sharp:=\sup\big\{ \lambda\in\R: \ep{\lambda}
    \text{ is solvable}\,\big\}.$$

However, due to the non-local character of the equation, it seems impossible
to prove that $\lambda_\sharp$ is finite, because such a result would require
a uniform control of the growth of possible solutions. Nevertheless, we have seen in
Subsection~\ref{sect:non.existence} that the growth of (super)solutions is
restricted somehow. Following this remark, let us define for $\mu>0$ the class
$$\Eclass(\mu):=\Big\{ u:\R^N\to\R: \limsup_{|x|\to\infty}
    \frac{u(x)}{\Psi(x)}\leq \mu \Big\}$$
where $\Psi$ has been defined in~\eqref{eq:def.psi.0}. Hence, instead of $\lambda_\sharp$, we will deal with
$$
\Lambda(\mu):=\big\{\lambda\in\R:
\text{there exists }  u\in\Eclass(\mu)\text{ solution of
    }\ep{\lambda}\,\big\}
$$
 and define a critical ergodic constant under this restrictive
growth $\lambda_*(\mu):=\sup\Lambda(\mu)$.
We will then relax the growth condition
in Section~\ref{sect:revisited}, after we have collected
more information on bounded from below solutions and uniqueness.

Let us set $\mu_0:=2+(\min f)^-$, so that $\mu_0=c_\lambda$ for
$\lambda=\min(f)$, see~\eqref{eq:super.lambda}. As we noticed, see Remark~\ref{rem:supersolution.bigger.constant},
for all $\mu\geq\mu_0$, $\mu\Psi$ is a strict supersolution of~\ep{} for $|x|\neq0$.

\begin{lemma}\label{lem:global.est}
    Assume that $u\in\Eclass(\mu)$ is a solution of
    \ref{eq:EP}  for some $\mu\geq\mu_0$. Then $u(x)\leq\mu\Psi+u(0)$ for all  $x\in\R^N$.
\end{lemma}

\begin{proof}
    We follow the same ideas as in the proof of Lemma~\ref{lemma:psiboundsw}.
    First, we define $\tilde{u}:=u-u(0)$, which is still a solution of
    \ref{eq:EP} such that $\tilde{u}\in\Eclass(\mu)$.

    By the limsup property, for any $\eta>0$ there exists
    $R_\eta$ such that for $|x|\geq R_\eta$, $\tilde u(x)\leq (1+\eta)\mu\Psi(x)$.
    Now, by continuity in $\overline{B}_{R_\eta}$, the maximum of
    $\tilde u-(1+\eta)\mu\Psi$ is attained at some point
    $x_0\in \overline{B}_{R_\eta}$. If the maximum is attained at the
    boundary, then $\tilde u\leq(1+\eta)\mu\Psi$ in $\R^N$.
    Similarly, if the maximum is attained at $x_0=0$ we have
    $\tilde u(0)=(1+\eta)\mu\Psi(0)=0$.

    Finally, if the maximum is attained at a point
    $x_0$ such that $0<|x_0|<R_\eta$, we use the comparison principle, Theorem~\ref{thm:comparison.principle}: $\tilde u$ is a
    (sub)solution of \ref{eq:EP} while $(1+\eta)\mu\Psi$ is a $C^1$-smooth, strict
    supersolution of \ref{eq:EP} and $\tilde u\leq(1+\eta)\mu\Psi$ outside the ball $B_{R_\eta}$. We reach a contradiction by using
    $(1+\eta)\mu\Psi$ as a test function for $\tilde u$ at $x_0$.

    The conclusion is that the maximum of $\tilde u-(1+\eta)\mu\Psi$ in $\R^N$
    is non-positive, and the result follows after letting $\eta$ tend to zero:
    $\tilde{u}\leq\mu\Psi$ which implies the estimate on $u$.
\end{proof}

\begin{lemma}
    Let $\lambda_1\in \Lambda(\mu)$ for some $\mu>0$.
    Then $\lambda_2\in\Lambda(\mu)$ for all $\lambda_2<\lambda_1$.
\end{lemma}
\begin{proof}
    Since $\lambda_1\in\Lambda(\mu)$, there exists a solution
    $u_1\in\Wloc(\mathbb{R}^N)\cap\Eclass(\mu)$ of
    $\ep{\lambda_1}$. But since $\lambda_2<\lambda_1$, it follows that
    $u_1-u_1(0)$ is a strict subsolution of \ep{\lambda_2}.
    Then, Theorem~\ref{thm.existence} with $\underline{u}=u_1-u_1(0)$ and
    $\lambda=\lambda_2$ yields a solution $u\in\Wloc(\mathbb{R}^N)$ of~\ep{\lambda_2}, such
    that $u(0)=0$ and $u\leq u_1$. Hence, $u\in\Eclass(\mu)$
    which implies that $\lambda_2\in\Lambda(\mu)$.
\end{proof}

\begin{lemma}\label{lem:lambda.finite}
   For any $\mu\geq \mu_0$, we have $\min(f)\leq\lambda^*(\mu)<\infty$.
\end{lemma}

\begin{proof}
We begin with the bound from below: from Corollary~\ref{cor:existence} we have that if $\lambda=\min(f)$, then $u_\lambda$ is a solution of~\ep{} such that $u_\lambda(0)=0$ and $u_\lambda\leq
c_\lambda\Psi$. But for this specific
$\lambda$, $c_\lambda=\mu_0$.  Hence, $u_\lambda$ belongs to
$\Eclass(\mu_0)\subset\Eclass(\mu)$ for any $\mu\geq \mu_0$, which proves that
for any $\mu\geq \mu_0$, $\lambda_*(\mu)\geq\min(f)$.

Assume now that $\lambda_*(\mu)=\infty$.
Then, there exists a sequence of solutions $\{(\lambda_n,v_n)\}$
  such that $\lambda_n\to\infty$ as
  $n\to\infty$, and thus we can assume that $\lambda_n\geq\min(f)$ for all $n$
  sufficiently large.

  Following \cite{BarlesMeireles2017}, we set
  $\psi_n:=\lambda_n^{-1/m}(v_n-v_n(0))$ so that
  $$
-\lambda_n^{1/m}\mathcal{L}[\psi_n]+\lambda_n|D\psi_n|^{m}=f-\lambda_n
  $$
  and after dividing by $\lambda_n$ we get
  \begin{equation}\label{ineq:contr.psi.n}
  |D\psi_n|^{m}=\lambda_n^{-1}f+\lambda_n^{1/m-1}\mathcal{L}[\psi_n]-1.
  \end{equation}

  Now we fix $R_0>0$ and use the \textit{implicit estimates} as in the proof of
  Lemma~\ref{lem:est.autocontrol}, but here we take into account the uniform
  estimate given by Lemma~\ref{lem:global.est} in order to control the convolution
  on $B_{R_0+1}\setminus B_{R_0}$: since $\psi_n(0)=0$ and $\lambda_n\geq1$ for
  $n$ big enough, we have $\psi_n\leq \lambda_n^{-1/m}\mu\Psi\leq\mu\Psi$. Hence
  $$
  \begin{aligned}
      \sup_{B_{R_0}} |D\psi_n|^m &\leq \lambda_n^{-1}\sup_{B_{R_0}}|f|
      +\lambda_n^{1/m-1}\sup_{B_{R_0}}|\,\mathcal{L}[\psi_n]\,|\\
      & \leq C_1(R_0)+\mu\sup_{B_{R_0+1}\setminus B_{R_0}} |\Psi|+
      2\sup_{B_{R_0}}|\psi_n|\,.
  \end{aligned}
  $$
  Recall that $\mu>0$ is fixed so that, setting
  $X:=\sup_{B_{R_0}}|D\psi_n|$, there exist some constants
  $a(R_0),\, b(R_0)$ such that for $n$ big enough
  $$X^m\leq a(R_0)+b(R_0)X.$$
  This yields a uniform bound ($i.e.$ independent of $n$) for the gradient of
  $\psi_n$ in any fixed ball $B_{R_0}$.

  Using the the fact that $\psi_n(0)=0$, up to extraction of a subsequence, we can assume that
  $\psi_n\to\psi$ locally uniformly for some $\psi\in\Wloc(\R^N)$. Then, sending
  $n\to\infty$ in \eqref{ineq:contr.psi.n}, we obtain a
  contradiction: $|D\psi|^m\leq -1$. The conclusion is that necessarily
  $\lambda_*(\mu)<\infty$.
\end{proof}

\begin{lemma}
    For any $\mu\geq \mu_0$, there exists a solution $v\in \Eclass(\mu)$ of
    \ep{\lambda} for the critical ergodic constant $\lambda=\lambda^*(\mu)$.
\end{lemma}

\begin{proof}
  Consider a sequence of solutions $\{\lambda_n,v_n\}$ such that
  $\lambda_n\to\lambda_*(\mu)$. Since for any $n$,
  $\tilde v_n:=v_n-v_n(0)\in\Eclass(\mu)$ and $\tilde{v}_n(0)=0$,
  this allows to use again the same \textit{implicit estimates}
  as in Lemma~\ref{lem:lambda.finite}. This implies that the sequence
  $\{\tilde v_n\}$ is locally uniformly bounded in $\Wloc(\mathbb{R}^N)$.
  Hence, up to extraction of a subsequence, we get local uniform convergence of
  $\tilde v_n$ to some $\tilde v\in\Wloc(\mathbb{R}^N)$ and passing to the limit in the
  viscosity sense, $\tilde v$ is a solution of $\ep{\lambda}$ for
  $\lambda=\lambda_*(\mu)$.
  Finally, since for any $n\in\N$, $\tilde v_n\leq \mu\Psi$, see Lemma~\ref{lem:global.est}, we have
  $\tilde v\in\Eclass(\mu)$ so that $\lambda_*(\mu)\in\Lambda(\mu)$.
\end{proof}

\section{Bounded from below solutions}
\label{sect:bounded}

Along this section we follow~\cite{Ichihara, Ichihara2013} but, as mentioned before, we have to adapt the arguments to take into account the non-local character of the problem. Here we use the specific sub and supersolutions that we constructed in
Section~\ref{sect:ass.pre} and $\Theta$ will stand for $\Theta_0$, the subsolution
constructed with $\lambda=0$ and $\kappa=1$, see Lemma~\ref{lem:subsol}.

Let us consider for $\sigma\in(0,1)$ the following equation, defined in
$\mathbb{R}^N$:
\begin{equation}
  \label{eq:with.sigmav}
  -\mathcal{L}[v]+|Dv|^m+\sigma v=f+ \sigma \Theta.
\end{equation}

\begin{lemma}
\label{lemma:phi0.sub}
There exists $c_1>0$ such that for any $\sigma>0$,
  $\theta_0:=\Theta-c_1\sigma^{-1}$ is a strict
  subsolution of~\eqref{eq:with.sigmav}.
\end{lemma}

\begin{proof}
From Lemma~\ref{lem:subsol} we see that $\Theta$ is not necessarily a
subsolution of~\ep{0}, since $f$ could be negative.
    However, there exists $c\geq0$ such that, in the viscosity sense,
    \begin{equation}\label{eq:approx.subsol}
        -\mathcal{L}[\Theta]+|D\Theta|^m-f\leq c\quad\text{in}\quad\R^N.
    \end{equation}
    Indeed, if $|x|>R_*$, using~\eqref{eq:theta.condition} with $\kappa=1$, then $-\mathcal{L}[\Theta]+
    |D\Theta|^m-f\leq 1-f\leq 0$, while for $|x|<R_*$,
    we have $-\mathcal{L}[\Theta]+ |D\Theta|^m-f \leq -\min(f)$. Finally,
    recall that (see Lemma~\ref{lem:subsol}) if $|x|=R_*$, no smooth
    test function can touch from above so that we do not need to check the
    subsolution condition in the viscosity sense.
    Hence,~\eqref{eq:approx.subsol} holds true with $c=(\min(f))^-\geq0$.

   Now, choosing $c_1>c$, we have that
  $$
  \begin{aligned}
  -\mathcal{L}[\theta_0]+|D\theta_0|+\sigma\theta_0-f& =
  -\mathcal{L}[\Theta]+|D\Theta|^m+\sigma(\Theta-c_1\sigma^{-1})-f\\
  &\leq c-c_1+\sigma\Theta< \sigma\Theta,
  \end{aligned}
  $$
  which proves the result.
\end{proof}

Now, in order to construct a supersolution to the viscous version of
\eqref{eq:with.sigmav} we use a $\C^2$-regularization of $\Psi$ as follows: let
$\overline\Psi\in\C^2(\mathbb{R}^N)$, such that $\overline{\Psi}=2\Psi$ if
$|x|\geq R_*$ and $\overline{\Psi}\geq 0$ if $|x|\leq R_*$. Notice that such a
$\overline{\Psi}$ exists, since $\Psi\geq 0$ in $\R^N$ and it is $\C^2$-regular for
$|x|\geq R_*$. Notice also that the constant $2$ corresponds to the choice
$\lambda=0$ in \eqref{eq:super.lambda}.

\begin{lemma}\label{lem:super.vsigma}
    There exists $c_2>0$ such that for any $R>R_*$,
    $\sigma>0$ and $0<\eps<\eps_0(c_2,R)$,
  $\psi_0:=\overline{\Psi}+c_2\sigma^{-1}$
  is a strict supersolution of
  \begin{equation*}\label{eq:with.sigmav.viscous}
      -\varepsilon\Delta v -\mathcal{L}[v]+|Dv|^m+\sigma v=
      f+ \sigma \Theta\ \text{ in } B_R\,.
  \end{equation*}
\end{lemma}

\begin{proof}
    Notice first that since for $|x|>R_*$, $\Theta(x)=(|x|-R_*)$, while
    $\Psi(x)=|x|f^{1/m}(x)\geq|x|$  and $\Psi\geq0$, then
    $\Theta\leq\overline{\Psi}$.
    Now, in a similar way as in the proof of Lemma~\ref{lemma:phi0.sub}, using  that $2\Psi$ is a strict supersolution of~\ep{} (with $\lambda=0$, see Proposition~\ref{prop:psilamba.strict}) for
     $|x|\geq R_*>0$, and
    that for $|x|\leq R_*$, both $f$ and $\overline\Psi$ are regular, we
    obtain $ -\mathcal{L}[\overline\Psi]+ |D\overline\Psi|^m-f\geq -c$
    for some $c>0$ independent of $\eps,\sigma,R$. Hence,
  $$
  \begin{aligned}
  -\eps\Delta\psi_0-\mathcal{L}[\psi_0]+|D\psi_0|+\sigma\psi_0-f&=
  -\eps\Delta\overline{\Psi}-\mathcal{L}[\overline\Psi]+|D\overline\Psi|^m+
  \sigma(\overline\Psi+c_2\sigma^{-1})-f\\
  &\geq -\eps\|\Delta\overline{\Psi}\|_{\L^\infty(B_R)}
  -c+c_2+\sigma\overline\Psi> \sigma\Theta,
\end{aligned}
  $$
  provided we choose $c_2>c$ and $\eps<\eps_0(c_2,R)$ small enough.
\end{proof}

\begin{lemma}
\label{lemma:existence.vsigma}
  For any $\sigma>0$, there exists a viscosity solution $v_\sigma\in\W_{\rm
      loc}(\mathbb{R}^N)$ of~\eqref{eq:with.sigmav}. Moreover, $\sigma
     v_\sigma(0)$ is bounded independently of $\sigma$.
\end{lemma}

\begin{proof}
 Using $\theta_0$ and $\psi_0$ as subsolution and supersolution respectively, we follow exactly
 the proofs of Proposition~\ref{proposition.existence.vR.R.varepsilon} and
 Theorem~\ref{thm.existence} to construct a solution with the desired
 properties. We sketch only the main modifications that we make here:

     \noindent $(i)$ We use $\theta_0$ as a strict
     subsolution in $B_R$, regularize it as $(\theta_0)_n$ to play the role of
     $\psi_n$ in Proposition~\ref{proposition.existence.vR.R.varepsilon}.
     This yields that for any $\varepsilon\in(0,1)$ and $R>1$ fixed,
     we have a solution $v_{\sigma,R,n,\eps}$ of the approximate problem
     \begin{equation}\label{eq:with.sigmav.approx}
      \left\{\begin{array}{ll}
      -\varepsilon\Delta v -\mathcal{L}_R^{\theta_0}[v]+|Dv|^m+\sigma v=
      f_n+ \sigma \Theta=\tilde{f}_n,& x\in B_R,\\
      v=\theta_0,&x\in\partial B_R.
      \end{array}\right.
     \end{equation}
     Notice that $\Theta$ is Lipschitz so that we do not need to regularize it
     in the right-hand side.

      \noindent $(ii)$
      As we already noticed, $\theta_0\leq\Theta\leq\overline{\Psi}$ in $\R^N$. Thus it follows that
      $-\mathcal{L}_R^{\theta_0}[\overline{\Psi}]\geq
      -\mathcal{L}[\overline{\Psi}]$. Hence, using Lemma~\ref{lem:super.vsigma}
      and the fact $\psi_0=\overline{\Psi}+c_2\sigma^{-1}\geq\overline{\Psi}$,
      we obtain that for $n$ big enough and $\eps$ small enough (depending on
      $c_2$ and $R$), $\psi_0$ is a supersolution
      of~\eqref{eq:with.sigmav.approx}.

      \noindent $(iii)$ Using the classical comparison result in $B_R$, see Theorem~\ref{thm:max.principle}, for
      $R>R_*$, we deduce that
      $$
        \Theta-\frac{c_1}{\sigma}=\theta_0\leq v_{\sigma,R,n,\eps}\leq \psi_0=
        \overline\Psi+\frac{c_2}{\sigma},
      $$
      which yields directly local uniform bounds for the solution,
      independent of $R>R_*$, $n$ and $\varepsilon$ (provided $R$ is fixed).
      Actually, this step is easier than in Theorem~\ref{thm.existence}.

      \noindent $(iv)$ Passing first to the limit as $\eps\to0$ (with $R>R_*$
      fixed), then as $n,R\to\infty$, we conclude that there exists a function
      $v_\sigma\in\W_{\rm loc}(\R^N)$, viscosity solution
      of~\eqref{eq:with.sigmav}, which verifies
      \begin{equation}
        \label{eq:bounds.for.vsigma}
        \Theta-\frac{c_1}{\sigma}\leq v_{\sigma}\leq \overline\Psi+\frac{c_2}{\sigma}.
      \end{equation}
      This implies that $\sigma v_\sigma(0)$ is bounded between $-c_1$ and
      $c_2$, which are constants independent of $\sigma$.
\end{proof}

The last step consists in sending $\sigma\to 0$ and get a  bounded from
below solution of~\ep{\lambda} for a certain $\lambda$.  We
define $w_\sigma:=v_\sigma-v_\sigma(0)$, which verifies $w_\sigma(0)=0$ and is a
viscosity solution of
\begin{equation}
  \label{eq:wsigma.solution}
  \lambda_\sigma-\mathcal{L}[w_\sigma]+|Dw_\sigma|^m+\sigma w_\sigma=f+ \sigma
  \Theta\ \text{ in $\mathbb{R}^N$},
\end{equation}
where $\lambda_\sigma=\sigma v_\sigma(0)$. Again, we need a uniform bound from
above in order to control the non-local term as $\sigma\to0$:
\begin{lemma}
\label{lemma:wsigma.bounded}
  There exists $\mu>0$ such that for any $\sigma\in(0,1)$, $w_\sigma\leq \mu\Psi$.
\end{lemma}

\begin{proof}
    The argument is similar to that of Lemma~\ref{lemma:psiboundsw}, except that
    we are in the whole space $\R^N$.

    Let us first notice that since $\lambda_\sigma$ is bounded from below, we
    can find $\mu>0$ such that for any $\sigma>0$,
    $c_{\lambda_\sigma}= 2+(\lambda_\sigma)^-<\mu$. This implies in particular
    that $\mu\Psi$ is a (strict) supersolution of $\ep{\lambda_\sigma}$.

    Now, we keep $\sigma>0$ fixed.
    Since $w_\sigma\leq 2\Psi+c_2\sigma^{-1}$, for $|x|$ big and $\mu>2$, it
    follows that $w_\sigma-\mu\Psi$ reaches a maximum at some point
    $x_0\in\R^N$.

 \noindent $(i)$ If $x_0=0$, the result follows by using that $w_\sigma(0)=0$
 and  $\mu\Psi\geq 0$.

  \noindent $(ii)$ Let $x_0\neq 0$. Up to a constant, we can assume that
  the maximum is such that $w_\sigma(x_0)>\mu\Psi(x_0)$. Otherwise we are done.
  Since $w_\sigma$ is a viscosity solution of~\eqref{eq:wsigma.solution}, we can
  use the subsolution condition at $x_0$ with $\mu\Psi$  as test
  function (recall that by construction $\mu\Psi$ is $\C^1$-smooth). We get
\begin{equation}
    \label{eq:wsigma.viscoity.sub}
    \lambda_{\sigma}-\mathcal{L}[w_\sigma]+|D(\mu\Psi)|^m+\sigma w_\sigma-f\leq
     \sigma \Theta\,.
\end{equation}
  Since $w_\sigma-\mu\Psi$ reaches a maximum at $x_0$, we have
  $-\mathcal{L}[w_\sigma](x_0)\geq -\mathcal{L}[\mu\Psi](x_0)$.
  Hence we get
  $$\lambda_\sigma-\mathcal{L}[\mu\Psi]+|D(\mu\Psi)|^m+\sigma \mu\Psi-f\leq
  \sigma \Theta\,.
  $$
  But since $\mu\Psi$ is a supersolution of $\ep{\lambda_\sigma}$ and
  $\mu\Psi>\Theta$, we reach a contradiction.

  The conclusion is that for any $\sigma>0$, we have $w_\sigma\leq\mu\Psi$ in
  $\R^N$ for some $\mu>0$ fixed.
  \end{proof}

In order to pass to the limit, we need local uniform estimates.
\begin{lemma}
  \label{lemma:wsigma.uniformly.bounded}
    Let $R_0>0$. For any  $R>R_0+1$,
    there exists a constant $C=C(R_0)$ such
    that $\|w_{\sigma}\|_{\W(B_{R_0})}\leq C$.
\end{lemma}

\begin{proof}
  We use the same \textit{implicit estimate} technique
  as in the proof of Lemma~\ref{lem:est.autocontrol} with only two minor
  modifications. The first one comes from the extra term $\sigma\Theta$ in
  equation~\eqref{eq:wsigma.viscoity.sub}, which does not pose any problem in $B_{R_0+1}$.
  The second comes from the non-local
  operator, which is defined now on the whole space \textit{i.e,} $\mathcal{L}$
  instead of $\mathcal{L}_R^{\psi}$. In order to  deal with this latter issue,
  we use the uniform bound $w\sigma\leq\mu\Psi$ on $B_{R_0+1}\setminus B_{R_0}$,
  see Lemma~\ref{lemma:wsigma.bounded}. Hence, the equivalent
  to~\eqref{eq:autocontrol1} reads now
  $$
  \label{eq:autocontrol.sigma}
  \begin{aligned}
      \sup_{B_{R_0}} |Dw_{\sigma}|^m &\leq \sup_{B_{R_0}}|f|+
      \sup_{B_{R_0}} |\sigma\Theta_0|+
      |\lambda_\sigma|+\sup_{B_{R_0}}\big|\mathcal{L}[w_{\sigma}]\big|\\
      &\leq C_0(R_0)+2\sup_{B_{R_0}}|w_{\sigma}|+
      \sup_{B_{R_0+1}\setminus B_{R_0}}|\mu\Psi|\\
      &\leq C_1(R_0) +C_2(R_0)\sup_{B_{R_0}}|Dw_{\sigma}|.
  \end{aligned}
  $$
  We conclude the proof as in Lemma~\ref{lem:est.autocontrol}, using that here
  also $w_\sigma(0)=0$ and get
  $\|w_{\sigma}\|_{\W(B_{R_0})} \leq C(R_0)$.
\end{proof}

Finally, we also need to control $w_\sigma$ uniformly from below:
\begin{lemma}
   \label{existence:bounded.from.below}
   There exists $M>0$ such that for any $\sigma>0$,
   \begin{equation*}
      \label{eq:compair.w.theta}
      w_\sigma\geq \Theta-M \quad \text{ in $\mathbb{R}^N$}.
   \end{equation*}
\end{lemma}

\begin{proof}
First of all, observe that, thanks to the estimate in $\Wloc$,
for fixed $R>R_*$, there exists $M=M(R)>0$ such that
$$\sup_{0<\sigma<1}\sup_{B_R}(|\Theta|+|w_\sigma|)\leq M\,.$$

In order to prove~\eqref{eq:compair.w.theta}, we fix $\delta\in(1/2,1)$ and show
that $w_\sigma\geq\delta\Theta-M$ in $\mathbb{R}^N$. We distinguish three cases:

\noindent$(i)$ If $|x|\leq R$, it is straightforward that
$\delta\Theta-w_\sigma\leq \sup_{B_R} (|\Theta|+|w_\sigma|)\leq M$, and the
result follows.

\noindent$(ii)$ Using~\eqref{eq:bounds.for.vsigma} we have that
$\inf_{\mathbb{R}^N}(w_\sigma-\Theta)>-\infty$. Then, since $M<\infty$ and
$\Theta\to\infty$ as $|x|\to\infty$, we get that
$$
  w_\sigma-\delta\Theta+M=(w_\sigma-\Theta)+(1-\delta)\Theta+M\to\infty
$$
as $|x|\to\infty$. Hence, there exists $R_1=R_1(\sigma,\delta)>R$,  such that
$w_\sigma\geq\delta\Theta-M$ for $|x|>R_1$.

\noindent$(iii)$ Finally, let $\mathcal{A}:=\{x\in\mathbb{R}^N : R<|x|<R_1\}$.
The idea here is to apply a comparison argument to the functions $w_\sigma$ and
$\delta\Theta-M$, which will imply the result. We observe that neither
$w_\sigma$ nor $(\delta\Theta-M)$ are a super or subsolution
of~\eqref{eq:with.sigmav}. But consider
\begin{equation}
  \label{eq:eq.modifies.sigma}
    -\mathcal{L}[\nu]+|D\nu|^m+\sigma \nu-f= \sigma \Theta-\sigma M\quad \text{ in
    }\mathcal{A}\,.
\end{equation}
Since $\lambda_\sigma=\sigma w_\sigma(0)$ is bounded by $\sigma M$,
from~\eqref{eq:wsigma.solution} we get that $w_\sigma$  is a supersolution
of~\eqref{eq:eq.modifies.sigma}.
On the other hand, since $R>R_*$, we have $\Theta=(|x|-R_*)$ and $f>1$. Hence,
using that $-\mathcal{L}[\Theta]\leq0$,
$$\begin{aligned}
  -\mathcal{L}[\delta\Theta-M] &+|D(\delta\Theta-M)|^m+
  \sigma(\delta\Theta-M)-f =-
  \delta\mathcal{L}[\Theta]+\delta^m+\sigma(\delta\Theta-M)-f\\
  &\leq \sigma\Theta+(\delta^m-\sigma M-1)\leq \sigma\Theta-\sigma M\,.
  \end{aligned}
  $$
  We conclude that $\delta\Theta-M$ is a subsolution
of~\eqref{eq:eq.modifies.sigma}.

Thanks to $(i)$ and $(ii)$ above, we have that
$w_\sigma \geq \delta\Theta-M$ on $\partial\mathcal{A}$ so we can apply a
comparison argument, see Theorem~\ref{thm:comparison.principle}, to get that $w_\sigma \geq \delta\Theta-M$ in $\mathcal{A}$.

From  $(i)$--$(iii)$ we have that $w_\sigma \geq \delta\Theta-M$ in
$\mathbb{R}^N$ and we conclude by letting $\delta\to 1$.
\end{proof}

We can finally prove the existence of a solution of~\ep{} that is bounded from below.
\begin{theorem}
  There exists a solution $(\lambda,u)$ of~\ep{} such that
  $\inf_{\mathbb{R}^N}(u-\Theta)>-\infty$ and $u\in\Eclass(\mu)$ for some
  $\mu>2$.
\end{theorem}

\begin{proof}
 In order to pass to the limit in~\eqref{eq:wsigma.solution}, we use  that
 $\sup_{\sigma}|\lambda_\sigma|$ is bounded independently of $\sigma$ and
 the bounds of Lemma~\ref{lemma:wsigma.uniformly.bounded}. This yields a sequence
 ${\sigma_n}$, with $\sigma_n\to 0$ as $n\to\infty$, a constant $\lambda$ and
 $u\in \W_{\rm loc}(\mathbb{R}^N)$, such that $\lambda_\sigma\to\lambda$ and
 $w_\sigma\to u$ as $n\to \infty$. By passing to the limit
 in~\eqref{eq:wsigma.solution} we also get that $(\lambda,u)$ is a solution
 of~\ep.

 Moreover, since $w_\sigma\geq\Theta-M$ for all $\sigma>0$ we get that
 $\inf_{\mathbb{R}^N}(u-\Theta)>-\infty$. Finally, we pass to the limit in
 $w_\sigma(0)=0$ and in the estimate $w_\sigma\leq\mu \Psi$ of
 Lemma~\ref{lemma:wsigma.bounded} to conclude that $u\in\Eclass(\mu)$ for
 some $\mu>2$.
\end{proof}

\section{Uniqueness}
\label{sect:uniqueness}

We are finally concerned with the uniqueness (up to addition of constants) of
solutions to~\ep{}, for $\lambda$ fixed.  To this aim,  we have to develop first
some tools related to the behaviour and comparison results of bounded from below
solutions.  Let us recall that we have constructed the supersolution
$\mu\Psi$ for large values of $x$, under the conditions \hyp{H0}--\hyp{H2}
and moreover, for that values of $x$, $\mu\Psi$ is just
$\mu|x|f^{1/m}$, see Section~\ref{sect:ass.pre}.

We begin with a lower-estimate. In
\cite[Proposition 3.4]{BarlesMeireles2017} the authors consider functions with power-type growth. Here we use a similar technique, but we have to refine it, due to the fact
that we are considering functions whose growth is far bigger than power-type.

\begin{lemma}
\label{lemma:u.bounded.below.Psi}
  Let $f$ verify \hyp{H0}--\hyp{H4} and let $u$ be a supersolution of~\ep{} for
  $|x|\gg1$ such that $u\in\Eclass(\mu)$ for some $\mu>0$ and
  $\inf_{\mathbb{R}^N} u>-\infty$.  Then, for any $\eta\in(0,\eta_0)$,
  there exists a  constant $C_\eta>0$, such that
  $$
    u(x)\geq C_\eta \Psi((1-\eta)x)-\frac{1}{C_\eta}
    \,,\text{ for }|x|\gg 1.$$
\end{lemma}

\begin{proof}
  Since $u$ is bounded from below, we can assume without loss of generality
  (just by adding a constant), that $u\geq 0$. Then we argue by contradiction;
  \textit{i.e.} we assume that there exists a sequence
  $|x_\varepsilon|\to\infty$ such that
  $$
    \frac{u(x_\varepsilon)}{\Psi((1-\eta)x_\varepsilon)}\to 0.
  $$
  Let $\alpha=1-\eta$ and define
  $$
  v_\varepsilon(s):=
  \frac{u(x_\varepsilon+s|x_\varepsilon|)}{\Psi(\alpha x_\varepsilon)},
  \quad s\in B_\eta.
  $$
  Then $v_\varepsilon(0)\to 0$ as $\varepsilon\to 0$ and
  $$
  |Dv_\varepsilon|^m(s)=
  \Big(\frac{|x_\varepsilon|}{\Psi(\alpha x_\varepsilon)}\Big)^m
  \Big|Du(x_\varepsilon+s|x_\varepsilon|)\Big|^m.
  $$
  We use now that $u$ is a non-negative supersolution of~\ep{} to estimate the
  gradient from below: $J\ast u\geq0$ so that
  $|Du(y)|^m\geq f(y)-u(y)-\lambda$. On the other hand, using
  Lemma~\ref{lem:global.est}, we have $u(y)\leq\mu\Psi(y)+u(0)$.

  We combine these inequalities at $y=x_\eps+s|x_\eps|$. Notice that
  $\alpha|x_\eps|\leq|y|\leq(1+\eta)|x_\eps|$, so that $|y|$ is large provided
  $\eps>0$ is small enough and in this case,
  $$
    \frac{|x_\varepsilon|}{\Psi(\alpha x_\varepsilon)}=
    \frac{1}{\alpha f^{1/m}(\alpha x_\varepsilon)}\,.
  $$
  Hence it follows that
  $$
  \begin{aligned}
    |Dv_\varepsilon|^m(s)& \geq \frac{1}{\alpha^mf(\alpha x_\varepsilon)}
    \Big(f(x_\varepsilon+s|x_\varepsilon|)-u(x_\varepsilon+s|x_\varepsilon|)
    -\lambda\Big)\\
    &\geq \frac{f(x_\varepsilon+s|x_\varepsilon|)}{\alpha^m
    f(\alpha x_\varepsilon)} -
    \frac{\mu\Psi(x_\varepsilon+s|x_\varepsilon|)+u(0)}{\alpha^m
    f(\alpha x_\varepsilon)}-\frac{\lambda}{\alpha^m
    f(\alpha x_\varepsilon)}\\
    & \geq \frac{f(x_\varepsilon+s|x_\varepsilon|)}{\alpha^m
        f(\alpha x_\varepsilon)}
    \Big(1-o_\eps(1)\frac{\mu}{\alpha^m}
    \frac{f(x_\varepsilon+s|x_\varepsilon|)}{f(\alpha x_\varepsilon)}
    +o_\varepsilon(1)\Big)\\
   & \geq \frac{C_\eta}{\alpha^m}(1-o_\varepsilon(1))
   \geq \frac{C_\eta}{2\alpha^m},\quad {\rm for\quad}\varepsilon\ll 1,
    \end{aligned}
  $$
  where the last two lines follow from \hyp{H3} and \hyp{H4}.
  Therefore, we conclude that for $\varepsilon$ small
  $$
    |Dv_\varepsilon|(s)\geq  \frac1{\alpha}(\frac{C_\eta}{2})^{1/m}>0.
  $$
  Let $w$ be a solution to $|Dw|=\frac1{\alpha}(\frac{C_\eta}{2})^{1/m}$  in
  $B_\eta$ with boundary data $ w=0$. By a standard comparison (in the viscosity
  sense) for the equation $|Du|=\text{constant}$, we deduce that
  $v_\varepsilon\geq w$ for any $\varepsilon\ll 1$. But this leads to a
  contradiction, since $w(0)>0$ while $v_\varepsilon(0)\to 0$.
\end{proof}

\begin{lemma}
\label{lemma:anu(ax).is.super}
    Let  $f$ verify~\hyp{H0}--\hyp{H3} and \hyp{H5}--\hyp{H6}.
    Let $u$ be a supersolution of~\ep{} such that $u\in\Eclass(\mu)$ for
    some $\mu>0$ and $\inf_{\mathbb{R}^N} u>-\infty$. Then,
    there exist $a_0>1$  and $R_1>0$ such that for all $a\in(1,a_0)$,
    $\overline{u}(x):= a^Nu(ax)$ is a  supersolution of~\ep{} for $|x|\geq R_1$.
\end{lemma}

\begin{proof}
    For simplicity, we reduce the computations to the case $u(0)=0$ and use the
    estimate $u\leq\mu\Psi$ given by Lemma~\ref{lem:global.est}.
    Define $\mathcal{A}[\overline{u}](x):=
    \lambda-\mathcal{L}[\overline{u}](x)+|D\overline{u}|^m(x)-f(x).$ Using that
    $u$ is a supersolution of~\ep{} and $a>1$, we have that
    \begin{equation}
    \label{eq:first.estimate.A}
    \begin{aligned}
      \mathcal{A}[\overline{u}](x)&=\lambda-\mathcal{L}[u](ax)+
      a^{(N+1)m}|Du|^m(ax)-f(ax)\\&
      \qquad+\mathcal{L}[u](ax) -\mathcal{L}[\overline{u}](x)+f(ax)-f(x)\\
      &\geq f(ax)-f(x)+\mathcal{L}[u](ax)-\mathcal{L}[\overline{u}](x).
    \end{aligned}
    \end{equation}
    Changing variables and using that $J$ is compactly supported in $B_1$, we
    can estimate the difference
    $\mathcal{L}[u](ax)-\mathcal{L}[\overline{u}](x)$ as follows
    $$
    \begin{aligned}
    \mathcal{L}[u](ax)-\mathcal{L}[\overline{u}](x)&\geq
    \int_{\mathbb{R}^N} J(ax-y)u(y)\d y-a^N\int_{\mathbb{R}^N} J(x-y)u(ay)\d y\\
    &= a^N\int_{\mathbb{R}^N} (J(az)-J(z))u(a(x-z))\d z\\
    &= a^N\int_{|z|<1} (J(az)-J(z))u(a(x-z))\d z.
    \end{aligned}
    $$
    Observe now that, since $J$ is radially decreasing, $J(az)<J(z)$ for all
    $a>1$ and moreover $ |J(az)-J(z)|\leq \|DJ\|_{\infty}(a-1)$, for $|z|<1$.
    Thus,
    $$
    \begin{aligned}
    \mathcal{L}[u](ax)-\mathcal{L}[\overline{u}](x)
    &\geq -a^N\|DJ\|_{\infty}(a-1)\int_{|z|<1} u(a(x-z))\d z\\
    &\geq -a^N\|DJ\|_{\infty}(a-1)\mu\sup_{|z|<1}\Psi(a(x+z)).
    \end{aligned}
    $$
    Plugin this estimate into~\eqref{eq:first.estimate.A} and using \hyp{H5} and
    \hyp{H6} we get
    \begin{equation}
    \label{eq:second.estimate.A}
    \begin{aligned}
    \mathcal{A}[\overline{u}](x)&\geq f(ax)-f(x)-
    a^N(a-1)\|DJ\|_{\infty}\mu\sup_{|z|<1}\Psi(a(x+z))\\
    &\geq af(x)-f(x)-a^N(a-1)\|DJ\|_{\infty}\mu f(x) o_x(1)
    \end{aligned}
    \end{equation}
    where $o_x(1)$ tends to $0$ as $x$ tends to infinity (this $o_x$ is uniform
    with respect to $a$).

    To conclude the proof take $|x|>R_1$ such that $\mu\|DJ\|_\infty a_0^N
    o_x(1)\leq 1/2$. Then~\eqref{eq:second.estimate.A} becomes
    $$
    \begin{aligned}
    \mathcal{A}[\overline{u}](x)&\geq
    f(x)(a-1)\Big(1-\mu\|DJ\|_\infty a_0^No_x(1) \Big)\geq f(x)\frac{a-1}{2}
    \geq 0
    \end{aligned}
    $$
    and $\overline{u}$ is a supersolution of~\ep{}.
   \end{proof}

Observe that Lemma~\ref{lemma:anu(ax).is.super} remains true independently of the hypothesis~\hyp{H7}. However, this Lemma will be not enough to prove the comparison result when $f$ has a \lq\lq slow\rq\rq\  growth, see the proof of Lemma~\ref{lemma:u1lequ2.fast} below. Therefore, we have to do a different approach in this latter case. Indeed, when  \hyp{H7}-slow holds, we use a similar argument as in the subquadratic
case of~\cite{BarlesMeireles2017}, proving that $u^q$ is a supersolution of~\ep{}. However, under our general hypotheses here,
we have to be more precise in the control of the constants, so that the computations are more tedious. Notice also that we still assume \hyp{H3}, which implies
that $f$ cannot grow too slow: typically faster than $f(x)=|x|^{m_*}$, where
$m_*=m/(m-1)\in(1,2)$.

The following estimate allows to control the non-local term in the case
of \hyp{H7}-slow. We get that
$\mathcal{L}[\Psi]$ can be controlled using the following estimate:
\begin{lemma}\label{lem:slow.est}
     Let $f$ verify \hyp{H0}--\hyp{H4} and \hyp{H7}-slow. Then there
    exists $C>0$ such that for all $|x|\gg1$,
    $$\sup_{B_1(x)}\Psi(y)\leq C\Psi(x).$$
\end{lemma}

\begin{proof}
    Fix $\eta\in(0,\eta_0)$. Then, for $|x|>1/\eta$,
    the ball $B_1(x)$ is contained in $\big\{x+s|x|:s\in B_\eta(0)\big\}$. Thus, using \hyp{H4} we
    have
    $$
      \sup_{B_1(x)}f(y)\leq\sup_{s\in B_\eta(0)}f(x+s|x|)\leq\csupeta
      f\big((1+\eta)x\big).
    $$
    Now, by \hyp{H7}-slow, there exists $C>0$ such that for $|x|$ big enough,
    $f\big((1+\eta)x\big)\leq Cf(x)$, which implies that
    $$
      \sup_{B_1(x)}f(y)\leq \csupeta C\,f(x).
    $$
    Finally, using that $\Psi(x)=|x|f^{1/m}(x)$ for $|x|$ large, we get the same
    result for $\Psi$ (with another constant).
\end{proof}

\begin{lemma}\label{lem:uq.super}
    Let $f$ verify \hyp{H0}--\hyp{H6} and \hyp{H7}-slow.  Let $u$ be a supersolution of~\ep{} such that $u\in\Eclass(\mu)$ for
    some $\mu>0$ and $\inf_{\mathbb{R}^N} u>-\infty$. Then,
    there exist $q_0>1$ and
    $R_1>0$ such that for all $q\in(1,q_0)$, the function $u^q$ is a supersolution of~\ep{} for  $|x|\geq R_1$.
\end{lemma}

\begin{proof}
    We first notice that under our assumptions, we can assume with no
    restriction that $u\geq0$ so that $|Du^q|^m=(qu^{q-1})^m|Du|^m$.
    Now take $\eta\in(0,\eta_0)$.
    By Lemma~\ref{lemma:u.bounded.below.Psi}
    there exists $R_1$ such that for
    $|x|>R_1$, $u\geq (C_\eta/2) \Psi=C_0\Psi$. Hence, for such $x$,
    $$
      |Du^q|^m\geq
      \big(q C_0^{q-1}\big)^m\Psi^{m(q-1)}|Du|^m
      \geq C_0^{m(q-1)}\Psi^{m(q-1)}|Du|^m\,.
    $$

    For the non-local term we
    use Lemma~\ref{lem:slow.est} as above and the fact that $u\in\Eclass(\mu)$:
    $$\begin{aligned}
        -\mathcal{L}[u^q] &= -\int J(x-y)u^q(y)\dy + u^q(x)\\
        &\geq -\mu^{q-1}\int J(x-y)\Psi^{q-1}(y)u(y)\dy + C_0^{q-1}\Psi^{q-1}(x)
        u(x)\\
        &\geq -(\mu C)^{q-1}\Psi^{q-1}(x)\int J(x-y)u(y)\dy+C_0^{q-1}\Psi^{q-1}(x)
        u(x)\\
        &\geq -C_1^{q-1}\Psi^{q-1}(x)(J\ast u)(x)+C_0^{q-1}\Psi^{q-1}(x)u(x)
    \end{aligned}$$
    where  $C_0,C_1$ are uniform with respect to $q\in(1,q_0)$.

    Using that $u$ is a supersolution of~\ep{}  to replace $(J\ast u)$ below we get
    \begin{equation}
      \label{eq:bounds.h7slow}
      \begin{aligned}
    \lambda-&\mathcal{L}[u^q]+|Du^q|^m\\ &\geq
    \lambda-C_1^{q-1}\Psi^{q-1}(J\ast u)+C_0^{q-1}\Psi^{q-1}u+
    C_0^{m(q-1)}\Psi^{m(q-1)}|Du|^m\\
    &\geq\lambda+C_1^{q-1}\Psi^{q-1}\Big(f-\lambda-|Du|^m\Big)
    +C_0^{m(q-1)} \Psi^{m(q-1)}|Du|^m\\
    &\qquad +\big(C_0^{q-1}-C_1^{q-1}\big)\Psi^{q-1}u\\
    &\geq \lambda+
    |Du|^m\big(C_0^{m(q-1)}\Psi^{m(q-1)}-C_1^{q-1}\Psi^{q-1}\big)\\
    &\qquad + C_1^{q-1}\Psi^{q-1}(f-\lambda)
    +(C_0^{q-1}-C_1^{q-1})\Psi^{q-1}u.
    \end{aligned}
    \end{equation}

If we choose $|x|$ big enough
    such that $C_0^m\Psi^m>2C_1\Psi$, which is possible since
    $\Psi$ is coercive, then
$(C_0^{m(q-1)}\Psi^{m(q-1)}-C_1^{q-1}\Psi^{q-1}\big)>\frac12 C_0^{m(q-1)}\Psi^{m(q-1)}$.

     Moreover,  since $u$ is a supersolution,
    $
      |Du|^m\geq f(x)-\lambda+J\ast u-u.
    $
    Using again Lemma~\ref{lem:slow.est} together with the fact that $u\leq\mu\Psi+u(0)$ we
    get $J\ast u\leq \mu C\Psi+u(0)$, which implies
    $
    |Du|^m\geq f-\lambda- C'\Psi(x)-u(0)
    $
    for some $C'>0$. Moreover, by \hyp{H3}, $\Psi(x)\ll f(x)$ so that for $|x|$ big
    enough, $|Du|^m(x)\geq f(x)/2$. We plug this into the last line in~\eqref{eq:bounds.h7slow} and get
    \begin{equation}
      \label{eq:bounds.h7slow2}
      \begin{aligned}
      \lambda-\mathcal{L}[u^q]+|Du^q|^m \geq \lambda &+
      \frac12 C_0^{m(q-1)}\Psi^{m(q-1)}f
    +  C_1^{q-1}\Psi^{q-1}(f-\lambda)\\
    &+(C_0^{q-1}-C_1^{q-1})\Psi^{q-1}u.
    \end{aligned}
    \end{equation}
      We first take $|x|$ big enough so that $f(x)>\lambda$ and $C_1\Psi(x)\geq1$.
    Then, by \hyp{H3},
    $u\leq\mu\Psi\leq f/2$ for $|x|$ big enough. And using again the fact that $\Psi$
    is coercive, for $|x|$ big enough we also have
    $$
    C_0^m\Psi^m\geq 2\big|C_0^{q-1}-C_1^{q-1}\big|^{1/(q-1)} \Psi.
    $$
    Thus, replacing in~\eqref{eq:bounds.h7slow2} yields
    $$
      \lambda-\mathcal{L}[u^q]+|Du^q|^m \geq \lambda+
      (f-\lambda)=f
    $$
    and the result holds.
\end{proof}

We are now ready to perform comparison results.

\begin{lemma}
\label{lemma:u1lequ2.fast}
    Let  $f$ verify~\hyp{H0}--\hyp{H7} (slow or fast)
    and let $u_1, u_2\in\Eclass(\mu)$ for
    some $\mu>0$ be respectively a subsolution and a supersolution of~\ep{} with
    $\inf_{\mathbb{R}^N}u_2>-\infty$. There exists $R_1>0$ such that if
    $u_1<u_2$ on $\partial B_{R_1}$, then $u_1\leq u_2$ in $B_{R_1}^C$.
\end{lemma}

\begin{proof}
    Let us begin by assuming \hyp{H7}-slow. By Lemma~\ref{lem:uq.super},
    for any $q>1$ (close enough to 1), $u_1^q$ is a supersolution in
    $\{|x|>1\}$. Since $u_1,u_2\in\Eclass(\mu)$, by
    Lemma~\ref{lemma:u.bounded.below.Psi} we see that $u_2^q\gg u_1$ as
    $|x|\to\infty$. So, the maximum of $u_1-u_2^q$ in $B_{R_1}^C$
    is attained at some point $x_0$. If $x_0\in\partial B_{R_1}$, then since
    $u_1(x_0)<u_2(x_0)$ we deduce the result. On the other hand, if $|x_0|>0$ we
    can use the equations and the Strong Maximum Principle,
    see Theorem~\ref{thm:max.principle}, to reach a contradiction.
    The conclusion is that for any $q\in(1,q_0)$, $u_1\leq u_2^q$ and the
    comparison follows by sending $q\to1$.

  Let us now turn to the case of \hyp{H7}-fast and
  define, for $a>1$, $\overline{u_2}:=a^Nu_2(ax)$. From
  Lemma~\ref{lemma:anu(ax).is.super}, we know that there exists $R_1>0$ and
  $a_0>0$,  such that $\overline{u_2}$  is a supersolution of~\ep{} in
  $B_{R_1}^C$, for any $a\in(1,a_0)$. Moreover if $a$ is chosen close enough to
  $1$, by continuity of both functions we have $\overline{u_2}\geq u_1$ in
  $\partial B_{R_1}$.

  Observe first that, by Lemma~\ref{lemma:u.bounded.below.Psi}, for any $a>1$
  fixed and $\eta\in(0,1)$, there exists a constant $c_{a,\eta}>0$ such that for
  $|x|$ large enough,
  $$
    \overline{u_2}(x)\geq c_{a,\eta}\Psi(a(1-\eta)x)-1/c_{a,\eta},
  $$
  while on the other hand, since $u_1\in\Eclass(\mu)$ we have
  $$
  u_1(x)\leq \mu\Psi(x)+u_1(0)\text{ in }\R^N.
  $$
  Therefore, for $|x|\gg 1$,
  $$
  \begin{aligned}
    (u_1-\overline{u_2})(x)&\leq \mu\Psi(x) -
    c_{a,\eta}\Psi(a(1-\eta)x)+1/c_{a,\eta}+u_1(0)\\
    &\leq
    \mu|x|\Big(f^{1/m}(x)-
    \frac{c_{a,\eta}}{\mu}a(1-\eta)f^{1/m}(a(1-\eta)x)\Big)+c(a,\eta,u_1).
  \end{aligned}
  $$
  Now, for $a\in(1, a_0)$ we fix $\eta>0$ small enough such that
  $a(1-\eta)>1$.
  Hypothesis \hyp{H7}-fast implies that $\liminf
  \big[ f^{1/m}(a(1-\eta)x)/f^{1/m}(x)\big]=+\infty$. Hence,
  for any $c>0$, there exists $c_0>0$ such that provided $|x|$ is big enough
  we have
  $$
  f^{1/m}(a(1-\eta)x)\geq c f^{1/m}(x)+c_0.
  $$
  From this, choosing conveniently $c$, it follows that as $|x|\to\infty$,
  $\limsup(u_1-\overline{u_2})= -\infty$. This implies that the supremum of
  $u_1-\overline{u_2}$ is attained at a point $x_0\in B_{R_1}^C$ or on the
  boundary $\partial B_{R_1}$.

  In the first case, \textit{i.e.} if $x_0\in B_{R_1}^C$, we get a contradiction
  by using the Maximum Principle (see Theorem~\ref{thm:max.principle}). In the
  second case, by assumption, $(u_1-\overline{u_2})(x)\leq
  (u_1-\overline{u_2})(x_0)\leq0$ for $x\in B_{R_1}^C$ and hence $u_1\leq
  \overline{u_2}$.

  The proof concludes by sending $a\searrow 1$, which implies  $u_1\leq {u_2}$ in
  $B_{R_1}^C$.
\end{proof}

The next two theorems show, not only the uniqueness of bounded from below
solutions, but also that this unique solution corresponds to the solution
associated with the critical ergodic constant $\lambda_*(\mu)$.

\begin{theorem}\label{thm:uniq.class.mu}
  Let  $f$ verify~\hyp{H0}--\hyp{H7} (slow or fast). Let $(\lambda_1,u_1)$ and
  $(\lambda_2,u_2)$ be two solutions of~\ep{\lambda_1} and~\ep{\lambda_2}, such
  that $u_1,u_2\in \Eclass(\mu)$ for some $\mu>0$
  and $\inf_{\mathbb{R}^N}u_1>-\infty$, $\inf_{\mathbb{R}^N}u_2>-\infty$. Then
  $\lambda_1=\lambda_2$ and $u_1=u_2+c$ for some constant $c\in\R$.
\end{theorem}

\begin{proof}
  Assume that $\lambda_2\leq \lambda_1$. Then $(\lambda_1,u_1)$ can be seen as a
  subsolution of~$\ep{\lambda_2}$. Moreover, by adding a constant, if
  necessary, we can ensure that $w=u_1-(u_2+C)$ verifies $\sup_{\partial
      B_{R_1}} w= 0$. Therefore, for any $\eps >0$, $u_1-(u_2+C+\eps)<0$ on
  $\partial B_{R_1}$ and we can apply
  Lemma~\ref{lemma:u1lequ2.fast} which gives that
  $u_1-(u_2+C+\eps)\leq0$ in $B_{R_1}^C$. After sending $\eps$ to zero, we get
  that $w\leq 0$ in $B_{R_1}^C$.

  Now,
  { if we consider the function $g(s)=|s|^m$, then, by convexity,
      $g(p+q)\geq g(p)+Dg(p)\cdot q$. Using this inequality with $p=Du_1$ and
      $q=Du_2-Du_1$ yields that $w$  is a subsolution of
      $(\lambda_1-\lambda_2)-\mathcal{L}[w]+c(x)|Dw|=0$, with $c(x)=m|Du_1|^{m-1}$.
      And since $\lambda_1\geq\lambda_2$, $w$ is a subsolution of
      $-\mathcal{L}[w]+c(x)|Dw|\leq0$}. Moreover, since $w\leq0$ outside $B_{R_1}$,
  we can use the comparison property in
  $B_{R_1}$ (see Theorem~\ref{thm:comparison.principle}) and deduce that $w\leq 0$ in $B_{R_1}$.

  Hence, $w$ reaches a maximum at some point in $\partial B_{R_1}$. 
  Thus, applying the Strong Maximum Principle, see
  Theorem~\ref{thm:max.principle}, we infer that $w=\max w$ in $\mathbb{R}^N$.
  This implies that $u_1=u_2+C$ and consequently, that $\lambda_1=\lambda_2$.
\end{proof}

\begin{theorem}
  Let $f$ verify~\hyp{H0}--\hyp{H7} (slow or fast) and let $(\lambda, u)$ be a solution
  of~\ep{} such that $u\in\Eclass(\mu)$ for some $\mu>0$ and
  $\inf_{\mathbb{R}^N}u>-\infty$. Then $\lambda=\lambda_*(\mu)$.
\end{theorem}

\begin{proof}
    The proof is done exactly as in \cite{BarlesMeireles2017} and uses arguments
    which are similar to that of Theorem~\ref{thm:uniq.class.mu}. Assume that
    $(\lambda,u)$ is a solution of \ep{} such that
    $u\in\Eclass(\mu)$ for some $\mu>0$ and $\inf u>-\infty$. Let also $v$ be a
    solution associated with the critical ergodic constant $\lambda_*(\mu)$.  We
    already know that $\lambda\leq\lambda_*(\mu)$, so we only need to prove the
    converse inequality.

    Take $R_1$ as in
    Lemma~\ref{lemma:u1lequ2.fast}. We can choose $C\in\R$ such that
    $\max_{\partial B_{R_1}}(v-(u+C))\leq0$. Considering the function
    $w:=\max(u+C+\eps,v)$,
    it turns out that $w$ is bounded from below because of $u$. It
    is also a subsolution of $\ep{\lambda_*(\mu)}$ because $u+C+\eps$ is a
    subsolution of this equation since $\lambda\leq\lambda_*(\mu)$. And moreover, $w\in\Eclass(\mu)$.

    Using Lemma~\ref{lemma:u1lequ2.fast}
    we deduce that for any $\eps >0$, $w\leq
    u+C+\eps$ in $B_{R_1}^C$, so that finally $w\leq u+C$ in $B_{R_1}^C$. The
    comparison in $B_{R_1}$ implies also that $w\leq u+C$ in $B_{R_1}$. Thus,
    $w-(u+C)$ reaches its maximum on $\partial B_{R_1}$ which implies that $w=u+C$
    in $\R^N$. The conclusion is that $v=u+C$, and thus $\lambda=\lambda_*(\mu)$.
\end{proof}

\section{Criticality revisited}
\label{sect:revisited}

In this section, we first extend the results of Section~\ref{sect:critical} on
critical ergodic constants to the more general class
$$
  \bar\Eclass:=\bigg\{u:\R^N\to\R: \limsup_{|x|\to\infty}
  \frac{u(x)}{\Psi(x)}<\infty\bigg\}.
$$
Notice that  $\bar\Eclass\supset\Eclass(\mu)$ for any $\mu>0$, and that
in $\bar\Eclass$, contrary to $\Eclass(\mu)$, we do not have a  uniform control of  the
behaviour of solutions (indeed for any $c>0$, $c\Psi\in\bar\Eclass$).

Let us define the critical ergodic constant in $\bar\Eclass$ as usual:
$$\bar\lambda:=\sup\big\{\lambda\in\R: \text{there exists } u\in\bar\Eclass,
    \text{ solution of }\ep{\lambda}\big\}.$$
We will prove here in particular that $\bar\lambda$ is finite.

\begin{lemma}\label{lem:bound.lambda.star}
    For any $\mu>\mu_0$, $\lambda_*(\mu)=\lambda_*(\mu_0)$.
\end{lemma}
\begin{proof}
    This is a consequence of the uniqueness and characterization of bounded from
    below solutions in $\Eclass(\mu)$. We know that (up to a constant)
    there exists a unique
    $u\in\Eclass(\mu_0)$ such that $\inf u>-\infty$ and $u$ is a solution of
    $\ep{\lambda_*(\mu_0)}$. Similarly, there is a unique $v\in\Eclass(\mu)$
    such that $\inf v>-\infty$ and $v$ is a solution of $\ep{\lambda_*(\mu)}$.
    Now, since $\Eclass(\mu_0)\subset\Eclass(\mu)$, we can apply
    Theorem~\ref{thm:uniq.class.mu} to conclude that $u=v$ (up to a constant) and
    $\lambda_*(\mu)=\lambda_*(\mu_0)$.
\end{proof}

\begin{corollary}
    Assume that $f$ satisfies \hyp{H0}--\hyp{H7}.
    Then $\min(f)\leq\bar\lambda<\infty$.
\end{corollary}

\begin{proof}
    Take any pair $(\lambda,u)$ solution of $\ep{}$ such that $u\in\bar\Eclass$.
    Since $u\in\Eclass(\mu)$ for some $\mu$, we deduce that, by definition of
    $\lambda_*(\mu)$,  $\lambda\leq\lambda_*(\mu)$. But using
    Lemma~\ref{lem:bound.lambda.star}, we get that necessarily
    $\lambda\leq\lambda_*(\mu_0)$. Hence,
    taking the supremum, $\bar\lambda\leq\lambda_*(\mu)<\infty$.
\end{proof}

Now, following \cite{BarlesMeireles2017}, let us give some Lipschitz estimate of
the critical ergodic constant $\bar\lambda$. We denote by $\bar\lambda(f)$ the
constant $\bar\lambda$ that corresponds to the equation with right-hand side
$f$. The following result extends \cite[Proposition 4.4]{BarlesMeireles2017} to more general
cases that we cover here.

\begin{lemma}
  Let $f_1$, $f_2$ verify~\hyp{H0}--\hyp{H2}. Assume that there exists a
  constant $c>0$ and a function $g>0$ such that $f_1(x),f_2(x)\geq cg(x)$ and
  $$ m:=\sup_{x\in\mathbb{R^N}}\frac{|f_1(x)-f_2(x)|}{g(x)}<\infty,$$
  then
  \begin{equation}\label{est.lip.lambda}
  |\bar\lambda(f_2)-\bar\lambda(f_1)|\leq
  \frac{m}{c+m}\max\{\bar\lambda(f_1),\bar\lambda(f_2)\}
  \end{equation}
\end{lemma}

\begin{proof}
  The proof is exactly the same as in~\cite{BarlesMeireles2017} and we omit it.
We only want to point out that, once we know that there is a solution
  to~\ep{\bar\lambda}, due to~\hyp{H0}--\hyp{H2}, the key points of the proof rely on the boundedness of $m$ and the lower
  bound for $f$ given by $g$.
\end{proof}

A typical application is to power-type functions $f$ as in
\cite{BarlesMeireles2017}, but we also have a similar result for faster growths,
for instance in the limiting case:
\begin{corollary}
    Assume that $f_i(x)\leq c_i\exp(m^{|x|})$, $i=1,2$, and define the
    function~$g$ as $g(x)=c_0\exp(m^{|x|})$ where $c_0:=\min(c_1,c_2)$. Then
    \eqref{est.lip.lambda} holds.
\end{corollary}

\

We end this section by a remark, more than a result, concerning the scaling
properties of $\bar\lambda$. Let $f(x)=|x|^\alpha$, with $\alpha>m_*$. Then, for
any $c>1$, it seems reasonable to think that
$$
\bar\lambda(cf)=c^{m_*N/(\alpha-m_*N)}\bar\lambda(f).
$$

The idea of the proof follows again~\cite{BarlesMeireles2017}. Our main
difficulty to complete the proof comes from the non-local term. Indeed, let
$u_1$ be a solution to~\ep{\bar\lambda}.  We would like to construct a solution
(or subsolution) to~\ep{} with right-hand side $\widetilde{f}=cf$. To this aim
consider $u_2(x)=a^{-\beta}u_1(ax)$, with $a=c^{1/(\alpha-m_*N)}<1$ and
$\beta=(N+m)/(m-1)$. The fact is that we are not able to prove that
$-\mathcal{L}[u_2](x)\leq -\mathcal{L}[u_1](ax)$ for all $x\in\mathbb{R}^N$, and
we only have, following the proof of Lemma~\ref{lemma:anu(ax).is.super}, that
$$
-\mathcal{L}[a^{N+\beta} u_2](x)\leq -\mathcal{L}[u_1](ax)+o(|x|^\alpha).
$$
Hence
$$
\begin{aligned}
-\mathcal{L}[u_2](x)+|Du_2(ax)|^m&\leq
a^{-m_*N}(-\mathcal{L}[u_1](ax)+|Du_1(ax)|^m)+o(|x|^\alpha)\\&=a^{-m_*N}(f(ax)-
\bar\lambda(f(ax)))+o(|x|^\alpha),
\end{aligned}
$$
which implies that,  $u_2$ is a subsolution of
$$
a^{-m_*N}\bar\lambda(a|x|^\alpha)-\mathcal{L}[u_2](x)+|Du_2(ax)|^m=
a^{-m_*N+\alpha}|x|^\alpha
$$
only for $x$ big enough. If we could prove that $u_2$ is a subsolution for all $x$
we would conclude, due to definition of $\bar\lambda$ as a supremum, that
$$
a^{-m_*N}\bar\lambda(a|x|^\alpha)\leq \bar\lambda(a^{-m_*N+\alpha}|x|^\alpha).
$$
In a similar way as in~\cite{BarlesMeireles2017} we could get the reverse
inequality. So, while it is not clear whether there is really a scaling property
for $\bar\lambda$ with power functions $f$, at least it seems that an
approximating scaling property should hold, with a different exponent than in
the local case.

\section*{Appendix}
\label{sect:appendix}

\setcounter{section}{0}
\renewcommand{\thesection}{\Alph{section}}

\section{Properties of the non-local operator $\mathcal{L}$}
Across the paper we use several times basic properties and technical estimates of $\mathcal{L}$ or $\mathcal{L}_R^\phi$. We summarize them here, for the reader's convenience.

Let $R>0$ and $\psi\in\C^0(\mathbb{R}^N)$. Recall that we have defined non-local operator $\mathcal{L}$ and the Dirichlet non-local operator, respectively, as
$$
\begin{aligned}
&\mathcal{L}[v](x):=\int_{\mathbb{R}^N} J(x-y)v(y)\d y-v(x),\\
& \mathcal{L}_R^\psi[v](x):=\int_{B_R} J(x-y)v(y)\d y+\int_{B_{R+1}\setminus B_R} J(x-y)\psi(y)\d y-v(x),
\end{aligned}
 $$
where $J$ is  a symmetric and compactly supported in $B_1$ kernel.

\begin{lemma}
Let $c$ be a positive constant and $u$, $v$ two positive functions. Then
\begin{enumerate}
\itemsep=2pt
  \item   $\mathcal{L}[c]=0$ and $\mathcal{L}_R^0[c]\leq 0$ if $c\geq 0$.
 \item $\mathcal{L}_R^\psi[u+c]=\mathcal{L}_R^\psi[u]+\mathcal{L}^0_R[c]$.
 \item $\mathcal{L}_R^\psi[\psi]=\mathcal{L}[\psi]$.
\item  $\mathcal{L}_R^\psi [u]\leq \mathcal{L}[u]$ if $\psi\leq u$.
\item    $\mathcal{L}[u](x_0)\leq \mathcal{L}[v](x_0)$ and $\mathcal{L}_R^\psi[u](x_0)\leq \mathcal{L}_R^\psi[v](x_0)$ if $u(x_0)=v(x_0)$ and $u\geq v$.
\item $\mathcal{L}[u](x)\leq \mathcal{L}[u](y)+o_{\delta}(1)$ and $\mathcal{L}_R^\psi[u](x)\leq \mathcal{L}_R^\psi[u](y)+o_{\delta}(1)$ if  $|x-y|^2\leq \delta$.
\end{enumerate}
\end{lemma}

We omit the proof, since  it follows straightforward from the definition of the non-local operators.

\begin{lemma}
  Let $x_0$ be a point where $u$ attains a positive maximum, respectively minimum. Then $\mathcal{L}[u](x_0)\leq 0$ and $\mathcal{L}_R^0[u](x_0)\leq 0$, respectively $\geq$.
\end{lemma}

\begin{proof}
  At the point $x_0$ where $u$ attains a positive maximum we have
  $$
  \mathcal{L}[u](x_0)=\int_{\mathbb{R}^N} J(x-y)u(y)\d y-u(x)\leq u(x_0)\Big(\int_{\mathbb{R}^N} J(x-y)-1\Big)=0.
  $$
We do a similar computation for $\mathcal{L}_R^0[u]$.
\end{proof}

\begin{lemma}
   If $g\in\C^1(\mathbb{R}^N)$ then $|\mathcal{L}[g](x)|\leq \sup\limits_{z\in
      B_1(x)}|Dg(z)|$.
\end{lemma}

\begin{proof}
  We use the fact that, for all $y\in B_1(x)$, $|g(x)-g(y)|\leq
  \sup\limits_{z\in B_1(x)}|Dg(z)|$. Then, by direct computation,  we obtain $$
  \begin{aligned}
    |\mathcal{L}[g](x)|\leq &\int_{B_1(x)} J(x-y)|g(y)-g(x)|\d y\leq
    \int_{B_1(x)} J(x-y)\sup\limits_{z\in B_1(x)}|Dg(z)|\d y\\
    =&\sup\limits_{z\in B_1(x)}|Dg(z)|,
  \end{aligned}
  $$
  since $J$ is
  compactly supported on $B_1$.
\end{proof}

\begin{lemma}
\label{lem:convex}
  Let $\psi$ be convex. Then $-\mathcal{L}[\psi]\leq 0$.
\end{lemma}

\begin{proof}
  The result follows from Jensen's inequality,
  $$
  -\mathcal{L}[\psi](x)=\int_{\mathbb{R}^N} \psi(x+z)\d \nu(z) -\psi(x) \leq
  \psi\Big(\int_{\mathbb{R}^N} (x+z)\d \nu(z)\Big) -\psi(x)\leq 0.
  $$
where $\nu$ denotes the  probability measure associated to $J$, $\d
  \nu(z)=J(z)\d z$.
\end{proof}

\begin{lemma}
  \label{lem:est.L.psi}
  Let $\psi$ be nondecreasing and
  for $\epsilon\in(0,1)$, let $c_\epsilon=\mu(B_1\setminus B_{1-\epsilon})$.
     Then \begin{equation*}
  \label{eq:upperbound.minusL}
-\psi(|x|+1|)+\psi(|x|)\leq  -\mathcal{L}[\psi](x)\leq -
c_\epsilon\psi(|x|+1-\epsilon)+\psi(|x|).\end{equation*}
\end{lemma}

\begin{proof}
Since $\psi$ is nondecreasing we have
$$
\begin{aligned}
   -\mathcal{L}[\psi](x)&=-\int_{B_1} J(y)(\psi(|x-y|)-\psi(|x|))\d y\geq -
   \int_{B_1} J(y)(\psi(|x|+1)-\psi(|x|)) \d y\\
&=-\psi(|x|+1)+\psi(|x|).
\end{aligned}$$

The other inequality yields as follows
$$  \begin{aligned}
   -\mathcal{L}[\psi](x)&=-\int_{B_1} J(y)(\psi(|x-y|)-\psi(|x|))\d y\leq
   -\int_{B_1\setminus B_{1-\epsilon}} J(y)\psi(|x-y|) \d y+ \psi(|x|)\\
&\leq -c_\epsilon\psi(|x|+1-\epsilon)+\psi(|x|).
\end{aligned}
$$
\end{proof}

\section{Comparison Results}
We prove here two comparison results that we use in several places across the paper. To this aim let us consider the general  equation
\begin{equation}
  \label{eq:max.principle.elin}
  -\mathcal{L}[w]+c(x)|Dw|+\alpha w=0,\quad \alpha\geq 0.
\end{equation}
Observe that this equation appears in different contexts. For instance, it turns out to be satisfied (with $\alpha=0$) by $w=v_1-v_2$ if  $v_1, v_2$ are a subsolution and a supersolution,
  respectively, of~\ep{}.

\begin{theorem}
    \label{thm:comparison.principle}
  Let $v\in\Wloc(\mathbb{R^N})$ be a subsolution of~{\rm\eqref{eq:max.principle.elin}}, such that for $R>1$,  $v\leq 0$ in
  $B_{R+1}\setminus B_R$. Then $v\leq 0$ in $B_R$.
\end{theorem}

\begin{proof}
  Let $x_0\in B_R$ be a point where $v$ reaches a positive maximum. Hence, the constant function $\varphi(x):=v(x_0)$ is an admissible test function for $v$ at $x_0$; i.e $v(x)-\varphi(x)$ reaches a maximum at $x_0$ and $v(x_0)=\varphi(x_0)>0$. Hence, since $|D\varphi|=0$ and $v\leq 0$ in $B_{R+1}\setminus B_R$
  $$
  \begin{aligned}
  0&\geq -\mathcal{L}[v](x_0)-c(x)|D\varphi(x_0)|+\alpha v(x_0)\\&
  = -\int_{B_R}J(x_0-y)v(y)\d y-\int_{B_{R+1}\setminus B_R}J(x_0-y)v(y)\d y+v(x_0)+\alpha v(x_0)
  \\ &\geq v(x_0)\Big(1-\int_{B_R}J(x_0-y)\d y +\alpha \Big)>0,
  \end{aligned}
  $$
  which is a contradiction. Hence $v(x_0)\leq 0$ and the result follows.
\end{proof}

\begin{remark}
  The result holds true even if we replace the gradient in~\eqref{eq:max.principle.elin}  by $c(x)|Dw|^{m-1}$, with $m>1$. Moreover, it is also true for the approximate problems that have a $-\varepsilon\Delta$-term. Indeed, it is straigthforward, since at a maximum point $-\varepsilon\Delta v(x_0)\geq 0$.
\end{remark}

\begin{theorem}[{\bf Strong Maximum Principle}]
\label{thm:max.principle}
  Let $v\in\Wloc(\mathbb{R^N})$ be a subsolution of~{\rm\eqref{eq:max.principle.elin}}, which reaches a maximum at $x_0\in\mathbb{R}^N$. Then $v\equiv v(x_0)$ in $\mathbb{R}^N$.
\end{theorem}

\begin{proof}
Let $x_0\in B_1$ be a point where $v$ reaches a maximum. As in the previous proof,  the constant function $\varphi(x):=v(x_0)$ is an admissible test function for $v$ at $x_0$. Hence, since $v(y)\leq v(x_0)$ for all $y\in B_1(x_0)$, we have
  $$
  0\geq -\mathcal{L}[v](x_0)-c(x)|D\varphi(x_0)|+\alpha v(x_0)
  \geq  -\int_{B_1(x_0)}J(x_0-y)(v(y)-v(x_0))\d y\geq 0.
  $$
 This implies that $v(y)=v(x_0)$ for all $y\in B_1(x_0)$.

We can repeat now the argument using as center any $y\in B_1(x_0)$ and get that $v(y)=v(x_0)$ for $y\in\mathbb{R}^N$.
 \end{proof}

\section{Existence result for an auxiliary problem}
We devote this last section of the Appendix to prove the existence of solutions of
equation~\eqref{eq.pbm.laplace.eps.ball}. To this aim we fix $\gamma\in(0,1)$,
$\varepsilon>0$, $R>1$ and we consider the following problem
\begin{equation}
  \label{eq.pbm.laplace.eps.ball.again.general}
   \left\{
   \begin{array}{ll}
       -\varepsilon\Delta\phi-\mathcal{L}_R^{\psi}[\phi]=f\,,\quad &
       x\in B_R,\\
        \phi=g& x\in \partial B_R\,,
 \end{array}
  \right.
\end{equation}
where $\psi\in\C^0(B_{R+1}\setminus B_R)$, $g\in\C^{0,\gamma}(\partial B_R)$ and
$f\in\C^{0,\gamma}(B_R)$.
\begin{lemma}
\label{lemma.existence.super.classical}
There exists a unique solution $\phi\in C^{2,\gamma}(B_R)\cap\C^0(\overline{B_R})$
of~\eqref{eq.pbm.laplace.eps.ball.again.general}.
\end{lemma}

Uniqueness comes from the comparison principle, see Theorem~\ref{thm:comparison.principle}. Actually we do a similar
argument in the proof of Lemma~\ref{lemma.supersolution.continuity}, so that we
skip the details here.

In order to prove the existence part of the result, we consider the unique function
$\varphi\in\C^{2,\gamma}(B_R)\cap\C^0(\overline{B_R})$ such that
of $-\Delta \varphi=0$ in $B_R$ with boundary data $\varphi=g$ on $\partial
B_R$ (see for instance \cite[Theorem 6.13]{GT}).
Then we set $\rho=\phi-\varphi$, which is a solution of
\begin{equation}
  \label{eq.pbm.laplace.eps.ball.again}
   \left\{
   \begin{array}{ll}
        -\varepsilon\Delta\rho-\mathcal{L}_R^{0}[\rho]=F,\quad&  x\in B_R,\\
        \rho=0& x\in \partial B_R,
 \end{array}
  \right.
  \end{equation}
  where
  $F:=f+\varepsilon\Delta\varphi+\mathcal{L}_R^\psi[\varphi]\in\C^{0,\gamma}(B_R)$.
  It is clear that, if $\rho\in C^{2,\gamma}(\overline{B_R})$ is a
  solution of~\eqref{eq.pbm.laplace.eps.ball.again}, then $\phi\in
  C^{2,\gamma}({B_R})\cap\C^0(\overline{B_R})$ and it verifies
  problem~\eqref{eq.pbm.laplace.eps.ball.again.general}. Notice that the
  boundary data for $\rho$ is zero, so that it belongs to
  $C^{2,\gamma}(\partial B_R)$.

  Hence we reduce the proof of Lemma~\ref{lemma.existence.super.classical} to
  proving existence for~\eqref{eq.pbm.laplace.eps.ball.again}.  This is based on
  the Continuity Method for elliptic operators (see~\cite[Theorem 5.2]{GT}) and
  two a priori bounds that we show next.

\begin{lemma}
\label{lemma.GT37.adapted}
  Let $\rho\in \C^{2}(\overline {B_R})$ be a solution
  to~\eqref{eq.pbm.laplace.eps.ball.again} in $B_R$. Then, there exists a constant
  $C=C(\varepsilon, R)$ such that
  \begin{equation} \label{eq:GT.bound1} \sup_{B_R}
    |\rho|\leq \sup_{\partial B_R}|\rho|+C\sup_{B_R}|F|
  \end{equation}
\end{lemma}

\begin{proof}
  Though the proof is essentially the same as~\cite[Lemma 3.7]{GT}, it has to be
  adapted carefully in some places in order to take into account the non-local
  term. To this aim, let $\alpha>0$, $\mathcal{L}_0=-\varepsilon\Delta
  -\mathcal{L}_R^0$ and define for $x_1\in[-R,R]$
  $$\widetilde\rho:=\sup_{\partial B_R}\rho^{+}+(\e^{\alpha R}-\e^{\alpha
      x_1})\sup_{B_R}(|F^+|/\varepsilon).$$

  We first observe that, if $\alpha=\alpha(\varepsilon, R)$ is chosen big
  enough, we get
  $$
  \mathcal{L}_0[\e^{\alpha x_1}]=-\alpha^2\varepsilon\e^{\alpha x_1}-
  \mathcal{L}_R^0[\e^{\alpha x_1}]\leq-\alpha^2\varepsilon\e^{\alpha x_1}+
  \e^{\alpha x_1}\leq -\varepsilon.
  $$
  Moreover, since for any constant $k\geq 0$, $\mathcal{L}_R^0[k]\leq 0$, we have,
  $$
  \mathcal{L}_0[\widetilde\rho]=\mathcal{L}_0[\sup_{\partial B_R}\rho^{+}+
  \e^{\alpha R}\sup_{B_R}(|F^+|/\varepsilon)]-\mathcal{L}_0[\e^{\alpha x_1}
  \sup_{B_R}(F^+/\varepsilon)]\geq  \sup_{B_R}|F^+|.
  $$
  Now, since $\mathcal{L}_0[\widetilde\rho-\rho]\geq \sup_{B_R}|F^+|-F\geq 0$ in
  $B_R$ and $\widetilde\rho-\rho\geq 0$ on $\partial B_R$, by the Maximum
  Principle, see~\cite[Theorem 6]{MelianRossi}, we get $\widetilde\rho-\rho\geq
  0$ in $B_R$, which yields
  $$
    \sup_{B_R} \rho\leq \sup_{\partial B_R} \rho^++C\sup_{B_R}F^+
  $$
  Replacing $\rho$ by $-\rho$ we get~\eqref{eq:GT.bound1}.
\end{proof}

\begin{lemma}\label{lemma.GT66.adapted}
  Let $\rho\in \C^{2,\alpha}(\overline{B_R})$ be a solution
  to~\eqref{eq.pbm.laplace.eps.ball.again}. Then, there exists a constant
  $C=C(N,\alpha,\varepsilon, R)>0$ such that
  $$
  \|\rho\|_{C^{2,\alpha}(\overline{B_R})}\leq
  C\|F\|_{C^{0,\alpha}(\overline{B_R})}
  $$
\end{lemma}

\begin{proof}
    Writing the equation as $\rho-\varepsilon\Delta
    \rho=\int_{B_R}J(x-y)\rho(y)\d y+F$ we
  get, using \cite[Theorem 6.6]{GT}, that there exists a constant $C_1>0$ such
  that
  $$ \|\rho\|_{C^{2,\alpha}(\overline{B_R})}\leq
  C_1\bigg(\|\rho\|_{C^{0}(\overline{B_R})}+\|F\|_{C^{0,\alpha}(\overline{B_R})}+
  \bigg|\int_{B_R}J(x-y)\rho(y)\d y\bigg|_{C^{0,\gamma}(\overline{B_R})}\bigg).
  $$
  On the other hand $\|\int_{B_R}J(x-y)\rho(y)\d y\|_{C^{0,\gamma}(\overline{B_R})}\leq
  C_2\|\rho\|_{C^{0}(\overline{B_R})}$, for some positive constant $C_2$.
  Combing these bounds with the previous one shown in
  Lemma~\ref{lemma.GT37.adapted}, we finally get
  $$
  \|\rho\|_{C^{2,\alpha}(\overline{B_R})}\leq C\|F\|_{C^{0,\alpha}(\overline{B_R})}
  $$
\end{proof}

\begin{proof}[Proof of Lemma~{\rm\ref{lemma.existence.super.classical}}]
  It is a direct adaptation of~\cite[Theorem 6.8]{GT}.
\end{proof}

\noindent{\sc Acknowledgments. ---} {Work partially supported by Spanish project MTM2014-57031-P.}


%

\end{document}